\documentclass[11pt]{article}

\usepackage{amsmath}
\usepackage{amsfonts}
\usepackage{amssymb}
\usepackage{amsthm}
\usepackage{amsopn}

\usepackage[margin=1in]{geometry}

\usepackage{hyperref}

\usepackage{caption}
\usepackage{subcaption}
\usepackage{cleveref}

\usepackage{graphicx}
\usepackage{xcolor,colortbl}

\newcommand{\graycell}{\cellcolor{gray!25}}
\usepackage{multirow}

\usepackage{algorithm}
\usepackage{algorithmic}

\usepackage{enumitem}
\usepackage{diagbox}
\setlist[enumerate]{leftmargin=.5in}
\setlist[itemize]{leftmargin=.5in}

\newcommand{\email}[1]{\protect\href{mailto:#1}{#1}}

\usepackage[utf8]{inputenc}
\usepackage[T1]{fontenc}
\usepackage[english]{babel}

\newtheorem{theorem}{Theorem}[section]
\newtheorem{remark}{Remark}[section]
\newtheorem{proposition}[theorem]{Proposition}
\newtheorem{corollary}[theorem]{Corollary}

\numberwithin{equation}{section}

\title{Multilevel Surrogate-based Control Variates%
\footnote{This work was funded by the PIA framework
(CGI, ANR) and the industrial members of the IRT
Saint Exupéry project R-Evol: Airbus, Liebherr,
Altran Technologies, Capgemini DEMS France,
CENAERO and Cerfacs.}%
}%

\author{Reda El Amri\footnote{IFP \'{E}nergies Nouvelles, Solaize, France (\email{mohamed-reda.el-amri@ifpen.fr}).}
\and Paul Mycek\footnote{CECI, CNRS - Cerfacs, Toulouse, France (\email{mycek@cerfacs.fr}, \email{ricci@cerfacs.fr}).} \footnotemark[4]
\and Sophie Ricci\footnotemark[3] \footnotemark[4]
\and Matthias De Lozzo\footnote{IRT Saint Exupéry, Toulouse, France (\email{matthias.delozzo@irt-saintexupery.com}).}}

\date{Technical report}

\usepackage[figuresright]{rotating}
\usepackage{booktabs}

\usepackage{multirow}

\usepackage[detect-weight=true]{siunitx}

\DeclareMathOperator{\Diag}{Diag}

\DeclareMathOperator*{\argmax}{arg\,max}

\DeclareMathOperator{\cost}{cost}

\newcommand{\E}{\mathbb{E}}
\newcommand{\V}{\mathbb{V}}
\newcommand{\C}{\mathbb{C}}
\newcommand{\M}{\mathbb{M}}
\newcommand{\R}{\mathbb{R}}

\renewcommand{\vec}[1]{\mathbf{#1}}
\newcommand{\mat}[1]{\mathbf{#1}}
\newcommand{\vx}{\vec{x}}
\newcommand{\vX}{\vec{X}}
\newcommand{\vZ}{\vec{Z}}

\newcommand{\bpercent}{\boldsymbol{\percent}}

\newboolean{hidedeletions}   
\setboolean{hidedeletions}{true}

\newcommand{\mdl}[1]{{\color{blue}#1}}
\newcommand{\mdlr}[1]{\ifthenelse{\boolean{hidedeletions}}{\vphantom{#1}}{\mdl{\st{#1}}}}

\hypersetup{
  pdftitle={Multilevel Surrogate-based Control Variates},
  pdfauthor={R. El Amri, P. Mycek, S. Ricci and M. De Lozzo}
}

\begin{document}

\maketitle

\begin{abstract}
Estimating statistical parameters of the output of a numerical simulator with random inputs by crude Monte Carlo sampling can be inaccurate when simulations are computationally expensive, so variance reduction techniques must be considered, such as the method of control variates (CVs) or the multilevel Monte Carlo (MLMC) method.
In this paper, we propose three multifidelity variance reduction strategies relying on surrogate-based CVs, possibly combined with MLMC. 
MLCV extends the CV method by using multiple auxiliary random variables devised from surrogate models for simulators of different levels. 
MLMC-CV improves the MLMC estimator by using a CVs based on a surrogate of the correction term at each level. 
Further variance reduction is achieved by using the surrogate-based CVs of all the levels in the MLMC-MLCV strategy.
Alternative solutions that reduce the subset of surrogates used for the multilevel estimation are also introduced.
These strategies are validated on an uncertain 1D heat equation test case from the MLMC literature, where the statistic is the mean integrated temperature at a given time.
The results are assessed in terms of the accuracy and computational cost of the multilevel estimators. 
It is shown that when the lower fidelity outputs are strongly correlated with the high-fidelity outputs, a significant variance reduction is obtained when using surrogate models for the coarser levels only. 
Taking advantage of pre-existing surrogate models proves to be an even more efficient strategy.\\

\end{abstract}


\section{Introduction}

In recent years, the propagation of uncertainties in numerical simulators has become an essential step in the study of physical phenomena. Therefore, uncertainty quantification (UQ) has emerged as an important element in scientific computing \cite{Smith2013_UQ, sullivan2015_UQ}. For complex nonlinear systems, the task of quantifying the effect of uncertainties on the simulator behaviour can pose major challenges as closed-form solutions often do not exist. 
Sampling-based algorithms are considered the default approach when it comes to UQ for complex nonlinear simulators. 
Monte Carlo (MC) sampling is arguably the most popular and flexible method.
It consists of extracting statistical information from a sample of simulator responses. 
Due to its non-intrusive nature, MC is straightforward to implement and provides unbiased estimators for classical statistical parameters, e.g., the expectation and the variance. 
Furthermore, the convergence of MC estimators in terms of their root-mean-square error (RMSE) as a function of the sample size $n$ is independent of the stochastic input dimension, but, on the other hand, it is very slow, namely at a rate of $n^{-0.5}$.
For computationally expensive simulators, the cost of MC is often considered impractically high. In some cases, slight improvements can be obtained through the use of importance sampling \cite{Hesterberg1995_IS}, latin hypercube sampling \cite{Mckay1979_LHS} or quasi-Monte Carlo (QMC) \cite{Dick2013_qMC}.

Another well-known approach in UQ consists in replacing the simulator by a surrogate model. Fast-to-evaluate surrogates can be used to approximate the high-fidelity model response and therefore reduce drastically the computational cost of estimating statistics. Polynomial chaos (PC) \cite{Luthen2021_SPCsurvey, Najm2009_UQPCcdf}, Gaussian process models \cite{Rasmussen2004_GPML}, radial basis functions \cite{Park1991_RBF} and (deep) neural networks \cite{Tripathy2018_DNN} are commonly used surrogate models. The downside of the surrogate-based approach is that it introduces approximation error, which causes biases in the statistics estimators. Besides, this approximation error tends to increase in high dimension.

One recent MC sampling framework based on control variates (CV) \cite{Lavenberg1977_ApplicationControlVariables, Lavenberg1981_PerspectiveUseControl, Lavenberg1982_StatisticalResultsControl, Lemieux2017_ControlVariates} has been extensively developed and used. In a sampling-based CV strategy, one seeks to reduce the variance of the MC estimator of a random variable, arising from the high-fidelity model, by exploiting its correlation with an auxiliary random variable that arises from low-fidelity models approximating the same input-output relationship. In classic CV theory, the mean of the auxiliary random variable is assumed to be known. Unfortunately, in many cases such an assumption is not valid. This creates the need to use another estimator for the auxiliary random variable \cite{Pasupathy2012_CVEM, Ng2014_MultifidelityApproachesOptimization, 
Peherstorfer2016_OptimalModelManagement,  Peherstorfer2018_SurveyMultifidelityMethods, 
Gorodetsky2020_GeneralizedApproximateControl, Schaden2020_MultilevelBestLinear, Schaden2021_AsymptoticAnalysisMultilevel, Schaden2021_thesis}, which involves an additional computational cost. Recently, the adoption of surrogate models as such auxiliary random variables, in particular PC models, has been explored in \cite{Fox2021_PCCV, Yang2022_ControlVariatePolynomial, Garg_PCCV, Zixi_PCCV}. The benefits are twofold: (1) the prediction of the surrogate models may be highly correlated with the output high-fidelity model, leading to a reduced variance of the CV estimator as compared to the MC estimation; (2) certain surrogate models provide exact statistics that are needed by the CV approach and for the others, these statistics can be approximated at negligible cost with arbitrary precision. 

Another variance reduction technique, called the multilevel MC (MLMC) method~\cite{Giles2008_MultilevelMonteCarlo, Cliffe2011_MultilevelMonteCarlo, Giles2015_MultilevelMonteCarloa}, uses a hierarchy of models. Originally devised for the estimation of expected values, MLMC has since been extended to the estimation of other statistics, see, e.g., \cite{Bierig2015_ConvergenceAnalysisMultilevel, Bierig2016_EstimationArbitraryOrder} for the estimation of variance and higher order central moments and \cite{Mycek2019} for the computation of Sobol' indices. Typically, multilevel methods are based on a sequence of levels which correspond to a hierarchy of simulators with increasing accuracy and cost. From a practical standpoint, the different levels often correspond to simulators with increasing mesh resolutions. This translates into a lower accuracy of the so-called \textit{coarse} levels, whereas the \textit{finer} levels correspond to accurate simulators. By construction, MLMC results in an unbiased estimator. It relies on a telescopic sum of terms based on the differences between the successive simulators. It debiases the MC estimator associated with the lowest-level simulator. Many other unbiased multilevel estimators have been devised in recent years. The multi-index MC estimator \cite{Haji-Ali2016_MultiindexMonteCarlo} is an extension of the MLMC estimator such that the telescoping sum idea is used in multiple directions. In \cite{Giles2009_MultilevelQuasiMonteCarlo}, the authors developed a variant of MLMC which uses QMC samples instead of independent MC samples for each level.

In this work, we propose new estimation techniques to quantify efficiently the output uncertainties of simulators with limited computational budgets. 
These approaches build on existing statistical techniques to combine the best of them: MC sampling to guarantee the unbiasedness of the estimators and the independence of their accuracy with respect to the input dimension, multilevel estimation to reduce their variance by exploiting some hierarchical fidelity structure, and surrogate-based CVs to futher reduce the variance, especially on the coarser levels, as efficiently as possible. 
It is worth noting that we do not seek to estimate statistics from accurate surrogate models, but rather from sampling techniques assisted by multifidelity simulators and coarse-to-fine surrogate models.
A first strategy, named \emph{multilevel control variates} (MLCV), combines multiple CVs based on surrogate models of simulators with different fidelity levels.
Although this strategy does not build on the MLMC approach specifically, it still exploits multilevel information through the surrogates constructed at different levels.
The second strategy, named \emph{multilevel Monte Carlo with control variates} (MLMC-MLCV), utilizes CVs in an information fusion framework to exploit synergies between the flexible MLMC sampling and the correlation shared between the high- and low-fidelity components. The numerical results show that the unbiased MLMC-MLCV estimators can converge faster than the existing estimators.

This technical report is organized as follows. 
We introduce notations and the necessary mathematical background in \cref{Background}. In \cref{multilevel}, we briefly discuss the MLMC estimator and then define our proposed MLMC-MLCV estimator that combines multilevel sampling and surrogate-based CV.  In \cref{sec:expe}, we conduct numerical experiments to support the theoretical results. \Cref{conclusion} proposes concluding remarks.


\section{Background}\label{Background}
In this section, we summarize the important results of statistical estimation using control variates and we show how surrogate models can be leveraged in this setting.

\subsection{Notation}

We first introduce a few notations. Throughout the rest of this paper, the high-fidelity numerical model is abstractly represented by the deterministic mapping
\begin{equation}
\begin{split}
  f: \Xi &\longrightarrow \mathbb{R} \\
  \vx &\longmapsto f(\vx),
\end{split}
\end{equation}
where $\vx := \begin{pmatrix} x_1 & \cdots & x_d\end{pmatrix}^\intercal$ is a vector of $d$ uncertain input parameters evolving in a measurable space denoted by $\Xi = \Xi_1 \times \cdots \times \Xi_d$, with $\Xi_i \subset \mathbb{R}$. 
Following a probabilistic approach, the $d$-dimensional input is modelled as a random vector 
${\vX \colon \Omega \to \R^d}$ with known \emph{joint} probability distribution.
In this work, we seek an accurate estimator of some statistic $\theta$ of the output random variable $Y:=f(\mathbf{\vX}) \colon \Omega \to \mathbb{R}$, e.g., its expected value ($\theta = \E[Y]$) or variance ($\theta = \V[Y]$), at a reasonable computational cost $\mathcal{C}$.

\subsection{Crude Monte Carlo}

In practice, the MC estimator $\hat{\theta}$ of the statistic $\theta$ is based on $n$ observations $Y^{(1)},\dots,Y^{(n)}$ defined as $Y^{(i)}=f(\vX^{(i)})$ where $\mathcal{X} = \{\vX^{(1)},\ldots,\vX^{(n)}\}$ is a $n$-sample of $\vX$, that is, a collection of independent and identically distributed random variables with the same distribution as $\vX$. For instance, the \emph{sample mean} and (unbiased) \emph{sample variance} estimators, 
\begin{equation}
    \hat{E}[Y] =  n^{-1} \sum_{i=1}^n Y^{(i)} 
    \quad \text{and} \quad 
    \hat{V}[Y] = (n-1)^{-1} \sum_{i=1}^n (Y^{(i)} - \hat{E}[Y])^2
\end{equation}
are MC estimators of the expectation and variance of $Y$, respectively.
It is well known that these estimators are \textit{unbiased}, that is, $\E[\hat{\theta}] = \theta$, so that the mean square error (MSE) of $\hat{\theta}$ reduces to the variance of the estimator:
\begin{equation}
\operatorname{MSE}(\hat{\theta},\theta) 
:= 
\E[(\hat{\theta} - \theta)^2] 
= 
\V[\hat{\theta}] + (\E[\hat{\theta}] - \theta)^2 = \V[\hat{\theta}].
\end{equation}
The convergence of these MC estimators is known to be slow, so that variance reduction techniques are needed when dealing with computationally expensive simulators.

\subsection{Control variates}\label{sec:CV}

In this section, we present a well-known variance reduction technique using auxiliary random variables $Z_1,\dots,Z_M$ as control variates. 
We denote by $\tau_1,\dots,\tau_M$ 
statistical parameters
related to the control variates $Z_1,\dots,Z_M$,
and assume they are known exactly. 
Then, the CV estimator is defined as
\begin{equation}\label{MCV}
    \hat{\theta}^{\textnormal{CV}}(\boldsymbol{\alpha}) = \hat{\theta} - \boldsymbol{\alpha^\intercal} (\hat{\boldsymbol{\tau}} - \boldsymbol{\tau}),
\end{equation}
where $\hat{\theta}$ and $\hat{\boldsymbol{\tau}} = \begin{pmatrix} \hat{\tau}_1 & \cdots & \hat{\tau}_M \end{pmatrix}^\intercal$ are \emph{unbiased} MC estimators of $\theta$ and $\boldsymbol{\tau} =\begin{pmatrix} \tau_1 & \cdots & \tau_M\end{pmatrix}^\intercal$, respectively, based on a \emph{common} input $n$-sample $\mathcal{X}$, and where $\boldsymbol{\alpha} \in \mathbb{R}^M$ is the control parameter. 
We note that the CV estimator is unbiased by construction, regardless of the value of the parameter $\boldsymbol{\alpha}$, and that its variance reads
\begin{equation}\label{varMCV}
    \V[\hat{\theta}^{\textnormal{CV}}(\boldsymbol{\alpha})] = \V[\hat{\theta}] + \boldsymbol{\alpha^\intercal} \boldsymbol{\Sigma} \boldsymbol{\alpha} - 2 \boldsymbol{\alpha^\intercal} \mathbf{c}.
\end{equation}
with $\mathbf{c}:=\C[\hat{\boldsymbol{\tau}},\hat{\theta}] \in \mathbb{R}^{M}$ and $\boldsymbol{\Sigma}:=\C[\hat{\boldsymbol{\tau}}] \in \mathbb{R}^{M \times M}$.
To fully take advantage of the control variates, the parameter $\boldsymbol{\alpha}$ is selected so as to minimize the variance (\ref{varMCV}) of the CV estimator. Assuming that the covariance matrix $\boldsymbol{\Sigma}$ is non-singular (and thus symmetric positive definite, SPD), the first- and second-order optimality conditions of the minimization problem imply that there exists a unique optimal solution,
\begin{equation}\label{eq:CV_alpha_star}
    \boldsymbol{\alpha}^* = \boldsymbol{\Sigma}^{-1} \mathbf{c}.
\end{equation}
The optimal CV estimator is thus $\hat{\theta}^{\textnormal{CV}}(\boldsymbol{\alpha^*})$, and its variance is given by
\begin{equation}\label{VRMCV}
    \V[\hat{\theta}^{\textnormal{CV}}(\boldsymbol{\alpha}^*)]  = (1 - R^2) \V[\hat{\theta}],
\end{equation}
where
$R^2 
= \V[\hat{\theta}]^{-1} \mathbf{c}^\intercal
\boldsymbol{\Sigma}^{-1} \mathbf{c}$
is nonnegative, as $\boldsymbol{\Sigma}^{-1}$ is SPD.
Denoting by $\mathbf{D} = \Diag(\boldsymbol{\Sigma}) \in \mathbb{R}^{M \times M}$ the diagonal matrix consisting of the diagonal of $\boldsymbol{\Sigma}$, it can be shown further that 
$R^2 
=
\mathbf{r}^\intercal
\mathbf{R}^{-1}
\mathbf{r}$,
where 
$\mathbf{r}
=
(\V[\hat{\theta}]\mathbf{D})^{-1/2} \mathbf{c}$ 
is the vector of the Pearson correlation coefficients between $\hat{\theta}$ and  $\hat{\boldsymbol{\tau}}$, and 
$\mathbf{R}
=
\mathbf{D}^{-1/2} \boldsymbol{\Sigma} \mathbf{D}^{-1/2}
$ 
is the correlation matrix of $\hat{\boldsymbol{\tau}}$.
Thus, $R^2 \in [0,1]$ corresponds to the squared coefficient of multiple correlation between $\hat{\theta}$ and the elements of $\hat{\boldsymbol{\tau}}$~\cite[section~2.5.2]{Anderson2003_IntroductionMultivariateStatistical}.
Consequently, $0 \leq \V[\hat{\theta}^{\textnormal{CV}}(\boldsymbol{\alpha^*})] \leq \V[\hat{\theta}]$, and we refer to 
$R^2$
as the variance reduction factor of the CV estimator.
This shows that the variance of the optimal CV estimator $\hat{\theta}^{\textnormal{CV}}(\boldsymbol{\alpha^*})$ is always reduced (or, rigorously speaking, not increased) compared to the MC estimator $\hat{\theta}$.
Furthermore, the higher $R^2$, the greater the reduction in variance.

\begin{remark}
    The requirement that $\boldsymbol{\Sigma}$ be non-singular is a reasonable one.
    Indeed, let us suppose that $\boldsymbol{\Sigma}$ is singular.
    Then, because $\boldsymbol{\Sigma}$ is positive semi-definite by construction, this implies that there exists a nonzero vector $\boldsymbol{\eta} \in \R^{M} \setminus \{\vec{0}\}$ such that $\boldsymbol{\eta}^\intercal \boldsymbol{\Sigma} \boldsymbol{\eta} = \V[\boldsymbol{\eta}^\intercal \hat{\boldsymbol{\tau}}] = 0$, indicating that any one element of $\hat{\boldsymbol{\tau}}$ can be expressed as an affine function of the others.
    As such, it does not bring any additional information to the CV estimator, so that at least one of the control variates can simply be discarded.
\end{remark}

A desirable property of the CV estimator is that increasing the number of control variates improves the CV estimator. Specifically, under mild assumptions, \cref{prop:add_cv} states that the variance of the CV estimator is reduced (or, rigorously speaking, not increased) when adding a new control variate.

\begin{proposition}\label{prop:add_cv}
Let 
$\hat{\boldsymbol{\tau}}_{+}
:= 
\begin{bmatrix}
    \hat{\boldsymbol{\tau}}^\intercal & \hat{\tau}_{M+1} 
\end{bmatrix}^\intercal
$,
$\boldsymbol{\tau}_{+}
:= 
\begin{bmatrix}
    \boldsymbol{\tau}^\intercal & \tau_{M+1} 
\end{bmatrix}^\intercal
$,
and define
$\mathbf{c}_+ = \C[\hat{\boldsymbol{\tau}}_{+}, \hat{\theta}]$, and $\boldsymbol{\Sigma}_{+} = \C[\hat{\boldsymbol{\tau}}_{+}]$.
We further assume that $\boldsymbol{\Sigma}$ and $\boldsymbol{\Sigma}_{+}$ are non-singular and that $\V[\hat{\tau}_{M+1}] > 0$.
Let 
$\hat{\theta}^{\textnormal{CV}}(\boldsymbol{\alpha^*_{+}})
:=
\hat{\theta} -
{\boldsymbol{\alpha}^*_+}^\intercal
(\hat{\boldsymbol{\tau}}_+ - \boldsymbol{\tau}_+)$,
with $\boldsymbol{\alpha}^*_+ = \boldsymbol{\Sigma}_+^{-1} \mathbf{c}_+$,
be the optimal CV estimator based on $M+1$ control variates, and let
$R^2_+$
denote its variance reduction factor.
Then 
$R^2_+ \geq R^2$.
\end{proposition}

\begin{proof}
We have
$R^2_+ 
= 
\mathbf{r}_+^\intercal \mathbf{R}_+^{-1} \mathbf{r}_+
$, where
$\mathbf{r}_+ = \begin{bmatrix}
    \mathbf{r}^\intercal & \gamma
\end{bmatrix}^\intercal$ and
$\mathbf{R}_+ = \begin{bmatrix}
    \mathbf{R} & \mathbf{u} \\
    \mathbf{u}^\intercal & 1
\end{bmatrix}$,
and where
$\mathbf{u} = (\V[\hat{\tau}_{M+1}] \mathbf{D})^{-1/2} \C[\hat{\boldsymbol{\tau}}, \hat{\tau}_{M+1}]$ and $\gamma = (\V[\hat{\theta}]\V[\hat{\tau}_{M+1}])^{-1/2}\C[\hat{\theta}, \hat{\tau}_{M+1}]$.
Note that $\mathbf{R}$ and $\mathbf{R}_+$ are both non-singular, because so are $\boldsymbol{\Sigma}$ and $\boldsymbol{\Sigma}_{+}$.
The augmented matrix $\mathbf{R}_+$ may be inverted by block using the Schur complement $s = 1 - \mathbf{u}^\intercal \mathbf{R}^{-1} \mathbf{u}\ne 0$ of the (1,1)-block $\mathbf{R}$.
It follows that $R^2_+ = R^2 + s^{-1}(\gamma - \mathbf{u}^\intercal \mathbf{R}^{-1} \mathbf{r})^2$.
Now, because $\mathbf{R}_+$ is SPD (it is positive semidefinite by construction and non-singular by assumption), $\vx^\intercal \mathbf{R}_+ \vx > 0$ for any choice of $\vx \ne \boldsymbol{0}$.
The particular choice of $\vx = \begin{bmatrix}
    - \mathbf{u}^\intercal \mathbf{R}^{-1} & 1
\end{bmatrix}^\intercal$ implies that $s > 0$, which in turn implies that $R^2_+ \geq R^2$.
\end{proof}

\begin{corollary}
Under the assumptions of \cref{prop:add_cv}, 
the equality
$R^2_+ = R^2$
holds if and only if $\gamma = \mathbf{u}^\intercal \mathbf{R}^{-1} \mathbf{r}$.
\end{corollary}

\begin{proof}
Straightforward from $R^2_+ = R^2 + s^{-1}(\gamma - \mathbf{u}^\intercal \mathbf{R}^{-1} \mathbf{r})^2$ with $s > 0$.
\end{proof}

For the CV estimation of specific statistics, the expressions of $\vec{c}$ and $\boldsymbol{\Sigma}$ may be further reduced to involve statistics on $Y$ and $\vZ := \begin{pmatrix}Z_1 & \cdots & Z_M\end{pmatrix}^\intercal$ directly. These expressions, as well as the consequences on \cref{eq:CV_alpha_star,VRMCV}, are given in \cref{app:CVmeanvar} for the CV estimators of $\E[Y]$ and $\V[Y]$.

It should be noted at this point that, in practice, the optimal parameter $\boldsymbol{\alpha}^*$ needs to be approximated.
Specifically, the statistics involved in $\vec{c}$ and $\boldsymbol{\Sigma}$ need to be estimated.
This can be done either using the same sample $\mathcal{X}$ as for the CV estimator itself, or using an independent, pilot sample $\mathcal{X}'$.
The former strategy introduces a bias in the resulting CV estimator, while the latter guarantees unbiasedness but requires additional high-fidelity simulator evaluations, so that the former strategy is generally preferred.
Furthermore, statistical remedies such as jackknifing, splitting or bootstrapping have been proposed to reduce or eliminate the bias introduced by the former strategy~\cite{Nelson1990_ControlVariateRemedies}.
In both strategies, however, neither the theoretical variance reduction given by \cref{VRMCV} nor \cref{prop:add_cv} hold anymore.

\begin{remark}
    In some instances, because it only involves the control variates, $\boldsymbol{\Sigma}$ may be known exactly (see \cref{rmk:Sigma_pce} for an example in a multilevel setting involving PC-based control variates) or estimated accurately using many independent samples at negligible cost. 
    Unfortunately, this is not the case for $\vec{c}$, which involves the output $Y=f(\vX)$ of the high-fidelity simulator $f$.
\end{remark}

\subsection{Surrogate-based control variates}\label{Surrogates}

The efficiency of the CV approach relies on the strong correlation between $Y=f(\vX)$ and the control variates $\vZ$. 
In a multifidelity framework, the control variates correspond to the output of low-fidelity versions of $f$, typically in the form of simulators with simplified physics and/or coarser discretization.
In many applications, the exact statistical measures $\boldsymbol{\tau}$ of such control variates $\mathbf{Z}$ may not be available.
One way to circumvent this limitation is to use additional samples to also estimate $\boldsymbol{\tau}$.
This type of estimators led to an \emph{approximate} CV class of methods~\cite{Gorodetsky2020_GeneralizedApproximateControl, Ng2014_MultifidelityApproachesOptimization, Peherstorfer2016_OptimalModelManagement}, where different model management and sample allocation strategies may be used to find a suitable trade-off between the additional cost of sampling $\vZ$ and the resulting variance reduction.
Recently, an optimal strategy was proposed~\cite{Schaden2020_MultilevelBestLinear, Schaden2021_AsymptoticAnalysisMultilevel, Schaden2021_thesis} leading to the so-called multilevel best linear unbiased estimator (MLBLUE), as well as a unifying framework for a large class of multilevel and multifidelity estimators.
However, ACV estimators typically yield lower variance reduction than their exact CV counterparts, as illustrated in \cref{app:cv_vs_acv} on bi-fidelity {(A)CV} estimators of the expectation.

In this paper, we focus on the case where the low-fidelity models correspond to surrogate models of the high-fidelity simulator $f$, which allows us to use the original CV approach described in \cref{sec:CV}.
Indeed, compared to the cost of evaluating the high-fidelity simulator $f$, the cost of evaluating surrogate models can be neglected, so that the statistical parameters $\boldsymbol{\tau}$ of the output of the surrogate models can be estimated arbitrarily accurately at negligible cost.
In fact, in some instances, the exact statistical parameters $\boldsymbol{\tau}$ may even be directly available, as detailed hereafter.
As a direct consequence, as discussed above, the question of sample allocation need not be addressed, and greater variance reduction is expected.
Using surrogate models in a CV strategy has been explored previously in \cite{Fox2021_PCCV, Yang2022_ControlVariatePolynomial}, where the available computational budget is allocated both to the construction of the surrogates and to the actual CV estimations.
In~\cite{Yang2022_ControlVariatePolynomial}, the authors introduce an approach that optimally balances the computational effort needed to select the optimal degree of the polynomial chaos (PC) expansion used in a stochastic Galerkin CV approach.
In our work, we focus on the situation in which the surrogate models are available, so that we do not consider any optimization strategy for the construction of the surrogates. 
Nevertheless, for the fairness of comparison, we still report the pre-processing cost of constructing the surrogate models.

We now briefly describe three different surrogate models commonly used in a UQ framework, and discuss the availability of statistics of their output.

\subsubsection{Gaussian process modelling}\label{sec:GP}
We assume that the simulator $f$ is a realization of a Gaussian process (GP) $F$ indexed by $\vx$ and defined by its mean function $m_F$ and covariance kernel $k_F$,
\begin{align}
\E[F(\vx)] &= m_F(\vx),\\
\C[F(\vx),F(\vx')] &= k_F(\vx,\vx'), \quad \forall \vx,\vx' \in \Xi.
\end{align} 
In practice, one can parametrize the forms of $m_F$ and $k_F$. 
For example, in the widely used ordinary GP method, a stationary GP is assumed. In this case, $m_F$ is set as a constant $m_F(\vx) = m$. 
More importantly, it is assumed that $k_F(\vx,\vx') = \bar{k}_F(\vx - \vx')$, and $k_F(\vx,\vx) = \bar{k}_F(0) = \sigma^2$ is a constant. 
Popular forms of kernels include polynomial, exponential, Gaussian, and Matérn functions. For example, the Gaussian kernel can be written as $k_F(\vx,\vx') = \sigma^2 \exp(-\frac{1}{2} \| \vx -\vx' \|^2_{\mathbf{h}})$, where the weighted norm is defined as $\| \vx -\vx' \|_{\mathbf{h}} = \left(\sum_{i=1}^{d} \frac{(x_i - x_i')^2}{h_i^2}\right)^{1/2}$ where $h_1,\ldots,h_d$ are correlation lengths. 
The hyperparameters $\sigma$ and $h_i$ can be obtained by maximum likelihood. Then, given $n$ observations $\mathcal{F} = \begin{pmatrix}f(\vx^{(1)}) & \cdots & f(\vx^{(n)}) \end{pmatrix}^\intercal$ 
of $F$ at
$\mathcal{X} = \begin{pmatrix} \vx^{(1)} & \cdots & \vx^{(n)} \end{pmatrix}^\intercal$, the posterior $\tilde{F}$ of $F$ can be defined as
\begin{equation}
\tilde{F} = [F \mid F(\mathcal{X}) = \mathcal{F}],
\end{equation}
whose expectation and covariance are given by
\begin{align}
m_{\tilde{F}}(\vx) & =\E[\tilde{F}(\vx)] = m_F(\vx) + k_F(\vx,\mathcal{X} )^\intercal k_F(\mathcal{X},\mathcal{X})^{-1}(\mathcal{F} - m_{F}(\mathcal{X})),\\
k_{\tilde{F}}(\vx,\vx') &= \C[\tilde{F}(\vx),\tilde{F}(\vx')] =  k_F(\vx,\vx')- k_F(\vx,\mathcal{X})^\intercal k_F(\mathcal{X},\mathcal{X})^{-1}k_F(\mathcal{X},\vx').
\end{align}
Thereafter, the high-fidelity model $f$ at $\vx$ will be approximated by the conditional expectation,
\begin{equation}
g^{\textnormal{GP}}(\vx) = m_{\tilde{F}}(\vx).
\end{equation}
Thus, the expectation and variance of $g^{\textnormal{GP}}(\vX)$ are defined as
\begin{align}
\E[g^{\textnormal{GP}}(\vX)] &= \int_{\Xi} m_{\tilde{F}}(\vx) p_{\vX}(\vx) \, \mathrm{d}\vx,\\ 
\V[g^{\textnormal{GP}}(\vX)] &= \int_\Xi (m_{\tilde{F}}(\vx) - \E[g^{\textnormal{GP}}(\vX)])^2 p_{\vX}(\vx) \, \mathrm{d}\vx.
\end{align}
These two statistics can be approximated empirically by taking a large sample of $\vX$, as the GP model is inexpensive. Analytical formulas exist for some pairs of distributions and covariance functions (see \cite[Table~1]{Briol2019_BayesQuad} and \cite{LeRiche2013_GPexpectation, elamri_OUU}).

\subsubsection{Taylor polynomials}\label{sec:taylor}
We assume that the numerical simulator $f$ is infinitely differentiable and that the moments of $\vX$ are finite. Under this assumption, it is possible to expand the original simulator $f$ around the input's expected value $\boldsymbol{\mu}_\vX := \E[\vX]$ as the infinite polynomial series according to Taylor's theorem,
\begin{equation}
    f(\vX) = g^{\textnormal{T}_\infty}(\vX) = \sum \limits_{|\boldsymbol{\beta}| \leq p} \frac{(\vX - \boldsymbol{\mu}_\vX)^{\boldsymbol{\beta}}}{\boldsymbol{\beta} !} D^{\boldsymbol{\beta}} f(\boldsymbol{\mu}_\vX),
\end{equation}
where $\boldsymbol{\beta} \in \mathbb{N}^d$, $|\boldsymbol{\beta}| = \sum_{i=1}^{d} \beta_i$, $\boldsymbol{\beta}! = \prod_{i=1}^{d} \beta_i !$, $x^{\boldsymbol{\beta}} = \prod_{i=1}^{d} x_i^{\beta_i}$ and $D^{\boldsymbol{\beta}} f =  \dfrac{\partial^{|\boldsymbol{\beta}|}f}{\partial^{\beta_1} x_1 \cdots \partial^{\beta_d} x_d}$. Thus, the function $f$ may be approximated by the first- and second-order Taylor polynomials,
\begin{align}\label{eq:Taylor1}
    &g^{\textnormal{T}_1}(\vX) = f(\boldsymbol{\mu}_\vX) + \mathbf{J}_f(\boldsymbol{\mu}_\vX) (\vX - \boldsymbol{\mu}_\vX), \\
    \label{eq:Taylor2}
    & g^{\textnormal{T}_2}(\vX) = f(\boldsymbol{\mu}_\vX) + \mathbf{J}_f(\boldsymbol{\mu}_\vX) (\vX - \boldsymbol{\mu}_\vX)  + \frac{1}{2} (\vX - \boldsymbol{\mu}_\vX)^\intercal \mathbf{H}_f(\boldsymbol{\mu}_\vX)(\vX - \boldsymbol{\mu}_\vX),
\end{align}
where $\mathbf{J}_f(\boldsymbol{\mu}_\vX) = \nabla_{\!f}(\boldsymbol{\mu}_\vX)^\intercal \in \mathbb{R}^{1 \times d}$ and $\mathbf{H}_f(\boldsymbol{\mu}_\vX) \in \mathbb{R}^{d\times d}$ are the Jacobian and Hessian matrices of $f$ at $\boldsymbol{\mu}_\vX$, respectively. In practical computations, the derivatives may be approximated by numerical differentiation if they are not provided. The expectation and variance of $g^{\textnormal{T}_1}(\vX)$ and $g^{\textnormal{T}_2}(\vX)$ are then defined as
\begin{align}
\E[g^{\textnormal{T}_1}(\vX)] &= f(\boldsymbol{\mu}_\vX),\\
\V[g^{\textnormal{T}_1}(\vX)] &= \mathbf{J}_f(\boldsymbol{\mu}_\vX)^{\odot 2} \boldsymbol{\sigma}^2_\vX,\\
\E[g^{\textnormal{T}_2}(\vX)] &= f(\boldsymbol{\mu}_\vX) + \frac{1}{2} \operatorname{Tr}(\mathbf{H}_f(\boldsymbol{\mu}_\vX)  \boldsymbol{\Sigma}^2_\vX),\\
\V[g^{\textnormal{T}_2}(\vX)] &= \operatorname{Tr}(\mathbf{J}_f(\boldsymbol{\mu}_\vX)^{\intercal} \mathbf{J}_f(\boldsymbol{\mu}_\vX) \boldsymbol{\Sigma}^2_\vX ) + \frac{1}{2}  \operatorname{Tr}( \mathbf{H}_f(\boldsymbol{\mu}_\vX)^{\odot 2} \boldsymbol{\sigma}^2_\vX (\boldsymbol{\sigma}^2_\vX)^{\intercal}),
\end{align}
where for any matrix $\mathbf{A}$,  $\mathbf{A}^{\odot 2}$ denotes the element-wise square of $\mathbf{A}$, $\operatorname{Tr}(\mathbf{A})$ is the trace of $\mathbf{A}$, $\boldsymbol{\sigma}^2_\vX = \begin{bmatrix}\sigma^2_{X_1} & \cdots & \sigma^2_{X_d}\end{bmatrix}^\intercal \in \mathbb{R}^d$ is the element-wise variance of $\vX$, and $\boldsymbol{\Sigma}^2_\vX = \Diag(\boldsymbol{\sigma}^2_\vX) \in \mathbb{R}^{d \times d}$.

\subsubsection{Polynomial chaos expansion}\label{sec:pce}

We consider the truncated polynomial chaos (PC) expansion of $f(\vX)$~\cite{Wiener1938_HomogeneousChaos, Cameron1947_OrthogonalDevelopmentNonlinear, Ghanem2003_StochasticFiniteElements, LeMaitre2010_SpectralMethodsUncertainty},
\begin{equation}
    f(\vX) \simeq g^{\textnormal{PC}_P}(\vX)
    = \sum_{k=0}^{P} \mathfrak{g}_k \Psi_k(\vX),
\end{equation}
where the coefficients $\mathfrak{g}_k$ are real scalars and $\Psi_k$ are orthonormal multivariate polynomials:
\begin{equation}
    \forall i,j \geq 0, \quad
    \langle \Psi_i, \Psi_j \rangle_{p_{\vX}}
    := \int_{\Xi} \Psi_i(\vx) \Psi_j(\vx) p_{\vX}(\vx) \, \mathrm{d}\vx
    = \E[\Psi_i(\vX)\Psi_j(\vX)]
    = \delta_{ij},
\end{equation}
where $\delta_{ij}$ denotes the Kronecker delta. The coefficients $\mathfrak{g}_k$ may be approximated using non-intrusive techniques such as non-intrusive (pseudo)spectral projection~\cite{Reagan2003_UncertaintyQuantificationReactingflow, Constantine2012_SparsePseudospectralApproximation, Conrad2013_AdaptiveSmolyakPseudospectral}, regression~\cite{Berveiller2006_StochasticFiniteElement, Sudret2008_GlobalSensitivityAnalysis, Blatman2010_EfficientComputationGlobal, Blatman2011_AdaptiveSparsePolynomial}, interpolation~\cite{Babuska2007_StochasticCollocationMethod, Nobile2008_AnisotropicSparseGrid}, or from intrusive approaches such as the stochastic Galerkin method~\cite{Ghanem2003_StochasticFiniteElements, LeMaitre2010_SpectralMethodsUncertainty}. Assuming $\Psi_0 \equiv 1$, the expectation and variance of $g^{\textnormal{PC}_P}(\vX)$ are then given by
\begin{equation}
    \E\left[g^{\textnormal{PC}_P}(\vX)\right] = \mathfrak{g}_0 \quad \mathrm{and}\quad
    \V\left[g^{\textnormal{PC}_P}(\vX)\right] = \sum_{k=1}^P \mathfrak{g}_k^2.
\end{equation}


\section{Multilevel estimators}\label{multilevel}
In this section, we present so-called \emph{multilevel} statistical estimation techniques based on a sequence of simulators $(f_\ell)_{\ell = 0}^L$, with increasing accuracy and computational cost, where $f_L\equiv f$ denotes the high-fidelity numerical simulator. 
The levels are ordered from the \emph{coarsest} ($\ell=0$) to the \emph{finest} ($\ell=L$).
We denote by $Y_\ell$ the random variable $Y_\ell=f_\ell(\vX)$ and $\theta_\ell$ the statistic of $f_\ell(\vX)$ increasingly close to $\theta_L\equiv \theta$. 

\subsection{Multilevel Monte Carlo}\label{sec:MLMC}

The statistic $\theta_L$ can be expressed as the telescoping sum
\begin{equation}\label{telescopingsum}
   \theta_L = \sum \limits_{\ell = 0}^{L} T_\ell,
\end{equation}
where $T_\ell = \theta_\ell - \theta_{\ell-1}$, and, by convention, $\theta_{-1} := 0$.
The MLMC estimator $\hat{\theta}_L^{\textnormal{MLMC}}$ of $\theta_L$ is then defined as~\cite{Giles2008_MultilevelMonteCarlo, Giles2015_MultilevelMonteCarloa, Teckentrup2013_MultilevelMonteCarlo}
\begin{equation}\label{MLMCestimator}
\hat{\theta}^{\textnormal{MLMC}} 
= 
\sum_{\ell = 0}^{L} \hat{T}_{\ell}^{(\ell)}, 
\end{equation}
where, at each level $\ell$, $\hat{T}_{\ell}^{(\ell)}$ is an unbiased MC estimator of $T_\ell$, based on an input sample $\mathcal{X}^{(\ell)} = \{\vX^{(\ell, i)}\}_{i=1}^{n_\ell}$ such that the members of $\mathcal{X}^{(\ell)}$ and $\mathcal{X}^{(\ell')}$ are mutually independent for $\ell \ne \ell'$.
In many instances, $\hat{T}_{\ell}^{(\ell)}$ is actually defined as $\hat{T}_{\ell}^{(\ell)} = \hat{\theta}_{\ell}^{(\ell)} - \hat{\theta}_{\ell-1}^{(\ell)}$, where $\hat{\theta}_{k}^{(\ell)}$ denotes the unbiased MC estimator of $\theta_{k}$ based on the simulator $f_k$ using the $n_{\ell}$-sample $\mathcal{X}^{(\ell)}$.

For example, the MLMC estimator $\hat{E}^{\textnormal{MLMC}}[Y]$ of the expectation $\E[Y]$ is given by
\begin{equation}
\hat{E}^{\textnormal{MLMC}}[Y] 
=
\hat{E}^{(0)} [Y_0]
+
\sum \limits_{\ell = 1}^{L} \hat{E}^{(\ell)} [Y_\ell]- \hat{E}^{(\ell)}[Y_{\ell-1}],
\end{equation}
where $\hat{E}^{(\ell)} [Y_{k}] = n_\ell^{-1} \sum_{i = 1}^{n_\ell} f_{k}(\vX^{(\ell,i)})$, with $k \in \{\ell, \ell-1\}$.
We stress that the correction terms at each level $\ell$ are computed from the same input sample $\mathcal{X}^{(\ell)}$, but using two successive simulators, $f_\ell$ and $f_{\ell-1}$. 
Similarly, the MLMC estimator $\hat{V}^{\textnormal{MLMC}}[Y]$ of the variance $\V[Y]$ is defined as~\cite{Bierig2015_ConvergenceAnalysisMultilevel}
\begin{equation}
\hat{V}^{\textnormal{MLMC}}[Y] 
=
\hat{V}^{(0)} [Y_0]
+
\sum \limits_{\ell = 1}^{L} \hat{V}^{(\ell)} [Y_\ell] - \hat{V}^{(\ell)}[Y_{\ell-1}],
\end{equation}
where $\hat{V}^{(\ell)} [Y_{k}] = \frac{n_\ell}{n_\ell - 1} ( \hat{E}^{(\ell)} [Y_{k}^2] - \hat{E}^{(\ell)} [Y_{k}]^2)$ is the single-level unbiased MC variance estimator.
Owing to the independence of the estimators $\hat{T}_{\ell}^{(\ell)}$, the variance of the MLMC estimator is
\begin{equation}\label{VarMLMC}
    \V[\hat{\theta}^{\textnormal{MLMC}}]
    = 
    \sum\limits_{\ell = 0}^{L} \V[\hat{T}_{\ell}^{(\ell)}].
\end{equation}

In practice, the MLMC method relies on the allocation of the total (expected) computational cost of the MLMC estimator across the different levels,
with 
\begin{equation}\label{eq:cost_MLMC}
\cost(\hat{\theta}^{\textnormal{MLMC}})
=
\sum_{\ell=0}^{L} n_\ell (\mathcal{C}_\ell+\mathcal{C}_{\ell-1}),
\end{equation}
where $\mathcal{C}_\ell$ is the (expected) cost of one evaluation of the simulator $f_{\ell}$, with $\mathcal{C}_{-1} := 0$ by convention. Thus, a key aspect is played by the choice of the number of samples $n_\ell$ allocated to each level $\ell$. The goal is to find the sample sizes $n_0,\ldots,n_L$ that minimize the variance of the estimator $\V[\hat{\theta}^{\textnormal{MLMC}}]$ given a computational budget $\mathcal{C}$. Thus, the sample allocation problem reads
\begin{equation}\label{eq:min_sample_alloc}
\begin{aligned}
& \operatorname*{minimize}\limits_{n_0,\ldots,n_L\in\mathbb{N}^*} && \V[\hat{\theta}^{\textnormal{MLMC}}] \\[.5em]
& \text{subject to} && \cost(\hat{\theta}^{\textnormal{MLMC}}) = \mathcal{C}.
\end{aligned}
\end{equation}
In practice, $\V[\hat{\theta}^{\textnormal{MLMC}}]$ is not known, and we instead rely on the assumption that
$\V[\hat{T}_{\ell}^{(\ell)}] \lesssim n_\ell^{-1} \mathcal{V}_\ell$, with $\mathcal{V}_\ell$ independent of $n_\ell$~\cite{Teckentrup2013_MultilevelMonteCarlo, Mycek2019}.
This is a reasonable assumption that holds for the MLMC estimation of the expectation, variance and covariance~\cite[Table~1]{Mycek2019}.
Note that, for the estimation of the expectation, we have $\V[\hat{T}_{\ell}^{(\ell)}] = n_\ell^{-1}\mathcal{V}_\ell$, with $\mathcal{V}_\ell = \V[Y_{\ell} - Y_{\ell-1}]$ and $Y_{-1} \equiv 0$.
The sample allocation problem \cref{eq:min_sample_alloc} is then replaced with
\begin{equation}\label{eq:min_sample_alloc_approx}
\begin{aligned}
& \operatorname*{minimize}\limits_{n_0,\ldots,n_L\in\mathbb{N}^*} && \sum_{\ell=0}^L n_\ell^{-1} \mathcal{V}_\ell \\[.5em]
& \text{subject to} && \cost(\hat{\theta}^{\textnormal{MLMC}}) = \mathcal{C},
\end{aligned}
\end{equation}
which is equivalent for the expectation, and an approximation for other statistics.
This minimization problem has a unique solution which can be computed analytically (see, e.g., \cite{Giles2008_MultilevelMonteCarlo, Cliffe2011_MultilevelMonteCarlo, Mycek2019}).


\subsection{Multilevel surrogate-based control variate strategies}
In this section, we introduce various surrogate-based control variate strategies in a multilevel framework where a hierarchy of simulators $(f_\ell)_{\ell = 0}^L$ is available. These strategies rely on using the random variables $(g_\ell(\vX))_{\ell = 0}^L$ as control variates, where $g_\ell$ is a surrogate model of the simulator $f_\ell$.

\subsubsection{Multilevel control variates (MLCV)}\label{sec:MCV}

The first strategy, referred to as \emph{multilevel control variates} and hereafter abbreviated MLCV, consists of using the surrogate models at \emph{all levels} to build the control variates in \cref{MCV}. Thus, $\tau_\ell$ corresponds to the statistical measure of the random variable $Z_\ell = g_\ell(\vX)$, and $\hat{\tau}_\ell$ to its unbiased MC estimator.
For instance, the MLCV estimator of the expectation based on an $n_L$-sample $\mathcal{X}^{(L)} = \{\vX^{(1)},\ldots,\vX^{(n_L)}\}$ reads
\begin{equation}\label{eq:MLCV}
    \hat{E}^{\textnormal{MLCV}}[Y](\boldsymbol{\alpha}) 
    = \hat{E}^{(L)}[Y_L]
    - 
    \sum_{\ell=0}^L \alpha_\ell \left(
        \hat{E}^{(L)}[Z_\ell] - \mu_{Z_\ell}
    \right),
\end{equation}
where $\hat{E}^{(L)}[Y_L] = n_L^{-1} \sum_{i=1}^{n_L} f_L(\vX^{(i)})$,  $\hat{E}^{(L)}[Z_\ell] = n_L^{-1} \sum_{i=1}^{n_L} g_\ell(\vX^{(i)})$,
and $\mu_{Z_\ell} = \E[g_\ell(\vX)]$.
Note that this approach does not build on the MLMC methodology described in \cref{sec:MLMC}, but still exploits multilevel information through the surrogates constructed at different levels.

\subsubsection{Multilevel Monte Carlo with control variates (MLMC-CV)}\label{sec:MLMF}

The second strategy consists of improving the MLMC estimator \cref{MLMCestimator} using the surrogate-based control variates $Z_0,\ldots,Z_L$.
Specifically, the MLMC-CV estimator improves the MC estimation of each of the correction terms of the MLMC estimator by using a surrogate model of the corresponding correction term as a control variate.
Note that the MLMF approach proposed in~\cite{Geraci2015_MultifidelityControlVariate, Geraci2017_MultifidelityMultilevelMonte} is based on a similar strategy, using arbitrary low-fidelity models at level in an \emph{approximate} CV setting.
The MLMC-CV estimator reads
\begin{equation}\label{eq:MLMF}
    \hat{\theta}^{\textnormal{MLMC-CV}}(\alpha_1, \ldots, \alpha_L) 
    = \sum_{\ell=0}^L 
    \hat{T}_{\ell}^{\textnormal{CV}}(\alpha_{\ell}),
\end{equation}
where $\hat{T}_{\ell}^{\textnormal{CV}}(\alpha_{\ell})$ is the CV estimator of the multilevel correction term $T_\ell$ (cf. \cref{telescopingsum}),
\begin{equation}\label{eq:correction_MLMF}
\hat{T}_{\ell}^{\textnormal{CV}}(\alpha_{\ell})
=
\hat{T}_{\ell}^{(\ell)}
- 
\alpha_{\ell}(\hat{U}^{(\ell)}_{\ell} - U_\ell),
\end{equation}
with 
$\hat{U}^{(\ell)}_{k}$ an unbiased MC estimator of 
the control variate statistic
$U_k = \tau_k - \tau_{k-1}$, based on the same input sample $\mathcal{X^{(\ell)}}$ as $\hat{T}_{\ell}^{(\ell)}$, again with members of $\mathcal{X^{(\ell)}}$ and $\mathcal{X^{(\ell')}}$ being independent for $\ell \ne \ell'$.
Note that, in \cref{eq:correction_MLMF}, because a single control variate is used at each level, the definition of $\hat{U}^{(\ell)}_{k}$ is used in the specific case where $k=\ell$. The more general definition when $k$ and $\ell$ are not necessarily equal will be useful later in \cref{sec:MLMC-MCV}, where we consider multiple control variates per level.
In practice, $\hat{U}^{(\ell)}_{k}$ may be defined as 
$\hat{U}^{(\ell)}_{k}
= \hat{\tau}^{(\ell)}_{k}
- \hat{\tau}^{(\ell)}_{k-1}$,
where $\hat{\tau}^{(\ell)}_{k}$ is an unbiased estimator of $\tau_k$ using the $n_\ell$-sample $\mathcal{X}^{(\ell)}$.

The optimal value $\alpha_\ell^*$ for $\alpha_\ell$ is obtained individually for each $\ell=0,\ldots,L$ as the optimal (single) CV parameter for $\hat{T}_{\ell}^{\textnormal{CV}}(\alpha_{\ell})$,
\begin{equation}
    \alpha_\ell^*
    =
    \dfrac{\C[\hat{T}_{\ell}^{(\ell)}, \hat{U}^{(\ell)}_{\ell}]}
    {\V[\hat{U}^{(\ell)}_{\ell}]},
\end{equation}
and the resulting variance of the control variate estimator of the correction is
\begin{equation}\label{eq:var_MLMF}
    \V[\hat{T}_{\ell}^{\textnormal{CV}}(\alpha_{\ell}^*)]
    =
    (1 - \rho_\ell^2) 
    \V[\hat{T}_{\ell}^{(\ell)}],
    \quad
    \mbox{with }
    \rho_\ell = 
    \dfrac{\C[\hat{T}_{\ell}^{(\ell)}, \hat{U}^{(\ell)}_{\ell}]}
    {\V[\hat{T}^{(\ell)}_{\ell}]^{1/2}\V[\hat{U}^{(\ell)}_{\ell}]^{1/2}}
    \in [-1, 1].
\end{equation}
The correction estimators $(\hat{T}_{\ell}^{\textnormal{CV}})_{\ell=0}^L$ being mutually independent,  the variance of the optimal MLMC-CV estimator is
\begin{equation}
    \V[\hat{\theta}^{\textnormal{MLMC-CV}}(\alpha_0^*,\ldots,\alpha_L^*)]
    =
    \sum_{\ell=0}^L
    \V[\hat{T}_{\ell}^{\textnormal{CV}}(\alpha_{\ell}^*)]
    =
    \sum_{\ell=0}^L
    (1 - \rho_\ell^2) 
    \V[\hat{T}_{\ell}^{(\ell)}]
    \leq
    \sum_{\ell=0}^L
    \V[\hat{T}_{\ell}^{(\ell)}],
\end{equation}
indicating that the variance of the MLMC-CV estimator is smaller (as long as $\rho_\ell^2 >0$) than that of the MLMC estimator; see \cref{VarMLMC}.
We remark that the variance reduction depends on the (squared) correlation between $\hat{T}_{\ell}^{(\ell)}$ and $\hat{U}^{(\ell)}_{\ell}$, which, in turn, typically relates to some measure of similarity between high-fidelity corrections $Y_\ell - Y_{\ell-1}$ and the corresponding CV corrections (see \cref{app:CVmeanvar} for the expectation and variance estimators in a single-level setting).

In our surrogate-based approach, we may define control variates as $Z_\ell = g_\ell(\vX)$, where $g_\ell$ is a surrogate of $f_\ell$, so that $\tau_\ell$ and $\tau_{\ell-1}$ could be estimated using samples of $Z_\ell$ and $Z_{\ell-1}$, respectively.
It is then crucial to construct the surrogates such that their successive differences $g_\ell - g_{\ell - 1}$ are good approximations of $f_\ell - f_{\ell - 1}$, to ensure a high similarity between $Y_\ell - Y_{\ell-1}$ and $Z_\ell - Z_{\ell-1}$.
However, constructing $g_\ell$ as a surrogate of $f_\ell$ does not give any guarantee on the quality of $g_\ell - g_{\ell-1}$ as a surrogate of $f_\ell - f_{\ell - 1}$.
Instead, in addition to the surrogates models $g_\ell$ of $f_\ell$, we construct surrogate models $h_\ell$ of the differences $f_\ell - f_{\ell-1}$, and we define auxiliary surrogate models $\tilde{g}_{\ell-1} = g_\ell - h_\ell$, for $\ell=1,\ldots, L$.
On each level $\ell$, we then use samples of $Z_\ell = g_\ell(\vX)$ and $\tilde{Z}_{\ell-1} = \tilde{g}_{\ell-1}(\vX)$ for the estimation of $\tau_\ell$ and $\tau_{\ell-1}$, respectively.
As a result, the variance reduction now depends on the similarity between $Y_\ell - Y_{\ell-1}$ and $W_\ell := Z_\ell - \tilde{Z}_{\ell-1} = g_\ell(\vX) - \tilde{g}_{\ell-1}(\vX) = h_{\ell}(\vX)$, where $h_\ell$ has been constructed to ensure such a similarity.
Specifically, for the expectation, the MLMC-CV estimator reads
\begin{multline}
    \hat{E}^{\textnormal{MLMC-CV}}[Y](\alpha_0,\ldots,\alpha_L) 
    = 
    \sum_{\ell=0}^L
    \left( \hat{E}^{(\ell)}[Y_\ell]
    - \hat{E}^{(\ell)}[Y_{\ell-1}] 
    \right) \\
    - \alpha_\ell \left(
        \hat{E}^{(\ell)}[Z_\ell]
        - \hat{E}^{(\ell)}[\tilde{Z}_{\ell-1}]
        -
        (\mu_{Z_\ell} - \mu_{\tilde{Z}_{\ell-1}})
    \right),
\end{multline}
with optimal values of $\alpha_\ell$ given by (see \cref{app:CVmeanvar-mean}, with $M=1$)
\begin{equation}
    \alpha_\ell^* = 
    \dfrac{\C[Y_\ell - Y_{\ell-1}, W_\ell]}
    {\V[Z_\ell - Z_{\ell-1}]},
    \quad \forall \ell=0,\ldots,L,
\end{equation}
resulting in level-dependent reduction factors $1-\rho_\ell^2$, where
\begin{equation}
    \rho_\ell^2
    =
    \dfrac{\C[Y_\ell - Y_{\ell-1}, W_\ell]^2}
    {\V[Y_\ell - Y_{\ell-1}]\V[W_\ell]}
\end{equation}
is the squared correlation coefficient between $Y_\ell - Y_{\ell-1}$ and $W_\ell = h_\ell(\vX)$.
It should be noted that, because of the linearity of the expectation operator and its MC estimator, the use of $\tilde{g}_{\ell-1}$ is superfluous.
Indeed, we may directly define the MLMC-CV estimator of the expectation as
\begin{equation}\label{eq:MLMC-CV_expect}
    \hat{E}^{\textnormal{MLMC-CV}}[Y](\alpha_0, \ldots, \alpha_L) 
    = 
    \sum_{\ell=0}^L
    ( \hat{E}^{(\ell)}[Y_\ell]
    - \hat{E}^{(\ell)}[Y_{\ell-1}] 
    )
    - \alpha_\ell (
        \hat{E}^{(\ell)}[W_\ell]
        -
        \mu_{W_\ell}
    ),
    \end{equation}
with $\mu_{W_\ell} = \E[W_\ell]$.
For the variance, the MLMC-CV estimator reads
\begin{multline}
    \hat{V}^{\textnormal{MLMC-CV}}[Y](\alpha_0,\ldots,\alpha_L) 
    = 
    \sum_{\ell=0}^L
    \left( \hat{V}^{(\ell)}[Y_\ell]
    - \hat{V}^{(\ell)}[Y_{\ell-1}] 
    \right) \\
    - \alpha_\ell \left(
        \hat{V}^{(\ell)}[Z_\ell]
        - \hat{V}^{(\ell)}[\tilde{Z}_{\ell-1}]
        -
        (\sigma^2_{Z_\ell} - \sigma^2_{\tilde{Z}_{\ell-1}})
    \right),
\end{multline}
with $\sigma^2_{Z_\ell} = \V[Z_\ell]$ and $\sigma^2_{\tilde{Z}_{\ell-1}} = \V[\tilde{Z}_{\ell-1}]$.
The resulting level-dependent reduction factors are then related to the correlation between $(Y_\ell - Y_{\ell-1} - \E[Y_\ell - Y_{\ell-1}])^2$ and $(h_\ell(\vX) - \E[h_\ell(\vX)])^2$ (see \cref{app:CVmeanvar-var}, with $M=1$).
Further details on the construction of $h_\ell$ will be given in \cref{sec:expe}.

\subsubsection{Multilevel Monte Carlo with multilevel control variates (MLMC-MLCV)}\label{sec:MLMC-MCV}

We now propose to further improve the MLMC-CV estimator \cref{eq:MLMF} by combining MLMC with the MLCV approach described in \cref{sec:MCV}, resulting in the MLMC-MLCV method.
The approach consists in using, at each level $\ell$, the surrogate-based control variates of \emph{all the levels} $\ell'=0,\ldots,L$, rather than only using those of level $\ell$, as was previously done with the MLMC-CV approach.

The MLMC-MLCV estimator then reads
\begin{equation}\label{estimatorMLCV}
\hat{\theta}^{\textnormal{MLMC-MLCV}} (\boldsymbol{\alpha}_{0}, \ldots, \boldsymbol{\alpha}_{L})
= \sum_{\ell = 0}^{L} \hat{T}_{\ell}^{\textnormal{MLCV}}(\boldsymbol{\alpha}_{\ell}),
\end{equation}
where
$\boldsymbol{\alpha}_{\ell}$ 
denotes the CV parameter at level $\ell$,
and $\hat{T}_{\ell}^{\textnormal{MLCV}}(\boldsymbol{\alpha}_{\ell})$ is the MLCV estimator of $T_\ell$,
\begin{equation}\label{correctionCV}
    \hat{T}_{\ell}^{\textnormal{MLCV}}(\boldsymbol{\alpha}_{\ell}) 
    = 
    \hat{T}_{\ell}^{(\ell)} 
    - 
    \boldsymbol{\alpha}_{\ell}^\intercal
    (\hat{\mathbf{U}}_\ell^{(\ell)} - \mathbf{U}_\ell),
\end{equation}
with
\begin{align}
    \mathbf{U}_0 &= (\tau_k)_{k=0}^L
    & \hat{\mathbf{U}}_0^{(0)} & = (\hat{\tau}_k^{(0)})_{k=0}^L \\
    \mathbf{U}_{\ell} & = (U_k)_{k=1}^L,
    \quad \text{for } \ell>0,
    & \hat{\mathbf{U}}_\ell^{(\ell)} & = (\hat{U}^{(\ell)}_{k})_{k=1}^L, 
    \quad \text{for }  \ell>0,
\end{align}
and with $U_k$ and $\hat{U}^{(\ell)}_{k}$ defined as in \cref{sec:MLMF}. 

Because each term \cref{correctionCV} is an unbiased (multiple) CV estimator of $T_\ell$, the resulting estimator \cref{estimatorMLCV} is also unbiased, and the optimal (variance minimizing) value $\boldsymbol{\alpha}_{\ell}^*$ of $\boldsymbol{\alpha}_{\ell}$ is given individually for each $\ell=0,\ldots,L$ as the optimal (multiple) CV parameter for $\hat{T}_{\ell}^{(\ell)}(\boldsymbol{\alpha}_{\ell})$,
\begin{equation}
\boldsymbol{\alpha}_{\ell}^* 
= \C[\hat{\mathbf{U}}_\ell^{(\ell)}]^{-1}
\C[\hat{\mathbf{U}}_\ell^{(\ell)}, \hat{T}_{\ell}^{(\ell)}].
\end{equation}
The resulting variance is given by 
\begin{equation}\label{eq:var_MLMC_MCV}
\V[\hat{\theta}^{\textnormal{MLMC-MLCV}}(\boldsymbol{\alpha}_{0}^*,\ldots,\boldsymbol{\alpha}_{L}^*)] 
 =
\sum_{\ell = 0}^{L} 
\V[\hat{T}_{\ell}^{\textnormal{MLCV}}(\boldsymbol{\alpha}_{\ell}^*)]
  = 
 \sum_{\ell = 0}^{L}
 (1 - R_\ell^2) \V[\hat{T}_{\ell}^{(\ell)}]
\end{equation}
with
\begin{equation}\label{eq:R2_MLMC-MLCV}
R_\ell^2 
=
\V[\hat{T}_{\ell}^{(\ell)}]^{-1}
\C[\hat{\mathbf{U}}^{(\ell)}, \hat{T}_{\ell}^{(\ell)}]^\intercal
\boldsymbol{\alpha}_{\ell}^*  \in [0,1]. 
\end{equation}
Again, owing to the fact that $R_\ell^2 \leq 1$, the variance of the MLMC-MLCV estimator
is always less than or equal to the variance of the MLMC estimator given by 
$\V[\hat{\theta}_L^{\textnormal{MLMC}}] = \sum_{\ell = 0}^{L} \V[\hat{T}_{\ell}^{(\ell)}]$
(see \cref{VarMLMC}). 

In our surrogate-based approach, $L+1$ surrogate-based control variates can be used for the coarsest level $\ell=0$, namely $g_0, \ldots, g_L$, so that $\boldsymbol{\alpha}_{0} \in \R^{L+1}$.
At correction levels $\ell >0$, we can use $L$ control variates based on $g_1, \ldots, g_L$ and $\tilde{g}_0, \ldots, \tilde{g}_{L-1}$, as described in \cref{sec:MLMF}, so that $\boldsymbol{\alpha}_{\ell} \in \R^{L}$, for $\ell>0$.

The MLMC-MLCV estimator of the expectation $\E[f(\vX)]$ is defined by \cref{estimatorMLCV}, with
\begin{align}
    \hat{T}_0^{\textnormal{MLCV}}(\boldsymbol{\alpha}_0)
    & = 
    \hat{E}^{(0)}[Y_0] - \boldsymbol{\alpha}_0^\intercal (\hat{E}^{(0)}[\vZ] - \boldsymbol{\mu}_{\vZ}), \\[.5em]
    \hat{T}_\ell^{\textnormal{MLCV}}(\boldsymbol{\alpha}_\ell)
    & = 
    \hat{E}^{(\ell)}[Y_\ell] - \hat{E}^{(\ell)}[Y_{\ell-1}] 
    - \boldsymbol{\alpha}_\ell^\intercal 
    (\hat{E}^{(\ell)}[\vec{W}] - \boldsymbol{\mu}_{\vec{W}}),
    \quad \mbox{for } \ell >0,
\end{align}
where 
\begin{align}
  \vZ &= (Z_0, \ldots, Z_L) = (g_0(\vX), \ldots, g_L(\vX)), & \boldsymbol{\mu}_\vZ &= \E[\vZ],\\[.5em]
  \vec{W} &= (W_1, \ldots, W_L) = (h_1(\vX), \ldots h_L(\vX)), & \boldsymbol{\mu}_{\vec{W}} &= \E[\vec{W}]. \label{eq:def_W}
\end{align}
The optimal values $\boldsymbol{\alpha}_\ell^*$ of the CV parameters are given by
\begin{align}
    \boldsymbol{\alpha}_0^* & = \boldsymbol{\Sigma}_0^{-1} \vec{c}_0,
    & \boldsymbol{\Sigma}_0 & = \C[\vZ],
    & \vec{c}_0 & = \C[\vZ, Y_0], \\
    \boldsymbol{\alpha}_\ell^* & = \boldsymbol{\Sigma}_\ell^{-1} \vec{c}_\ell,
    & \boldsymbol{\Sigma}_\ell & = \C[\vec{W}],
    & \vec{c}_\ell & = \C[\vec{W}, Y_\ell - Y_{\ell-1}], 
    \quad \mbox{for } \ell >0,
\end{align}
resulting in $R^2_\ell = \V[Y_\ell - Y_{\ell-1}]^{-1}\vec{c}_\ell^\intercal \boldsymbol{\Sigma}_\ell \vec{c}_\ell$, for $\ell = 0, \ldots, L$.
Note that $\boldsymbol{\Sigma}_\ell = \C[\vec{W}]$ is the same for all $\ell > 0$.
The optimal MLMC-MLCV variance estimator is derived in \cref{app:MLCVmeanvar}.

\begin{remark}\label{rmk:Sigma_pce}
In practice, $\boldsymbol{\Sigma}$ and $\vec{c}_{\ell}$ may be estimated using either a pilot sample or the same sample as for the estimation of $\hat{\mathbf{U}}_\ell^{(\ell)}$ and $\hat{T}_{\ell}^{(\ell)}$. 
Alternatively, in the specific context of PC-based control variates, for the estimation of the expectation, a closed-form expression for $\boldsymbol{\Sigma}$ can be obtained. 
Letting $Z_\ell = g_\ell(\vX) = \sum_{k=0}^{P_g^\ell} \mathfrak{g}_{\ell, k} \Psi_k(\vX)$ and $W_\ell = h_\ell(\vX) = \sum_{k=0}^{P_h^\ell} \mathfrak{h}_{\ell, k} \Psi_k(\vX)$, we have
\begin{align}
    \forall m, m' = 0,\ldots,L, \quad
    [\boldsymbol{\Sigma}_0]_{m, m'}
    & = \C[Z_m, Z_{m'}]
    = \sum_{k=1}^{\min(P_g^m, P_g^{m'})} 
    \mathfrak{g}_{m,k}  \mathfrak{g}_{m',k},\\
    \forall m, m' = 1,\ldots,L, \quad
    [\boldsymbol{\Sigma}_\ell]_{m, m'}
    & = \C[W_m, W_{m'}]
    = \sum_{k=1}^{\min(P_h^m, P_h^{m'})} 
    \mathfrak{h}_{m,k} \mathfrak{h}_{m',k},
    \quad \mbox{for } \ell >0.
\end{align}
\end{remark}

\subsection{Practical details}\label{sec:practical_details}
A summary of the methods is presented in \cref{tabsummary_methods}, including two minor variations of the MLMC-CV and MLMC-MLCV estimators introduced previously, which consist of using fewer surrogate models for the multilevel estimation.
We remark that all the MLMC-like estimators (including the MLMC estimator), hereafter abbreviated MLMC-*, have variance $\sum_{\ell = 0}^{L}
 (1 - R_\ell^2) \V[\hat{T}_{\ell}^{(\ell)}]$, with 
 \begin{itemize}
     \item $R^2_\ell = 0$ for plain MLMC;
     \item $R^2_\ell = \rho^2_\ell$ as defined in \cref{eq:var_MLMF} for MLMC-CV; and 
     \item $R^2_\ell$ defined by \cref{eq:R2_MLMC-MLCV} for MLMC-MLCV.
 \end{itemize}
 For all these methods, we will further assume that 
$\V[\hat{T}_{\ell}^{(\ell)}] \lesssim n_\ell^{-1} \mathcal{V}_\ell$
(see \cref{sec:MLMC}), which
implies that 
$\V[\hat{T}_{\ell}^{\textnormal{MLMC-*}}(\boldsymbol{\alpha}_{\ell}^*)]
=
(1 - R_\ell^2) \V[\hat{T}_{\ell}^{(\ell)}]
\lesssim
n_\ell^{-1} \mathcal{V}^{\textnormal{CV}}_\ell$,
with $\mathcal{V}^{\textnormal{CV}}_\ell := (1 - R_\ell^2)\mathcal{V}_\ell$.

\begin{sidewaystable}[!htbp]
    \setlength{\defaultaddspace}{.5em}
    \centering
    \small
    \caption{Summary of the methods. The first three are state-of-the art methods, the next three are the novel multilevel methods proposed in this paper, and the last two are variants of the proposed methods.}
    \label{tabsummary_methods}
    \begin{tabular}{p{9.5cm}ll}
    \toprule
Method & Form of the estimator & Eq. \\
\midrule
\textbf{Monte Carlo (MC)}. 
&
$\hat{\theta}$ &
\\
\addlinespace
\textbf{Control Variates (CV)}~\cite{Lavenberg1977_ApplicationControlVariables, Lavenberg1981_PerspectiveUseControl, Lavenberg1982_StatisticalResultsControl, Nelson1990_ControlVariateRemedies, Lemieux2017_ControlVariates}.
& 
$
\hat{\theta} - \sum\limits_{m=1}^M \alpha_m (\hat{\tau}_m - \tau_m)$. &
\ref{MCV}
\\
\addlinespace
\textbf{Multilevel Monte Carlo (MLMC)}~\cite{Giles2008_MultilevelMonteCarlo, Giles2015_MultilevelMonteCarloa}. 
& 
$
\hat{\theta}_{0}^{(0)} + \sum\limits_{\ell=1}^{L} \hat{T}_{\ell}^{(\ell)}$. &
\ref{MLMCestimator}
\\
\addlinespace
\midrule
\textbf{Multilevel Control Variates (MLCV)}. 
CV method where the CVs are based on surrogate models of simulators of different levels of fidelity. & 
$
\hat{\theta} - \sum\limits_{\ell=0}^L \alpha_\ell (\hat{\tau}_\ell - \tau_\ell)$. &
\ref{MCV}
\\
\addlinespace
\textbf{MLMC-CV}. MLMC with one CV at each correction level based on surrogate models of the simulators on the corresponding level.
Corresponds to a surrogate-based MLMF~\cite{Geraci2015_MultifidelityControlVariate, Geraci2017_MultifidelityMultilevelMonte} where the exact statistics of the CVs are known. & 
$
\hat{\theta}_{0}^{(0)} - \alpha_0 (\hat{\tau}_0^{(0)} - \tau_0)
+ \sum\limits_{\ell=1}^{L} \left(
\hat{T}_{\ell}^{(\ell)} - \alpha_\ell (\hat{U}_\ell^{(\ell)} - U_\ell)
\right)$. &
\ref{eq:MLMF},~\ref{eq:correction_MLMF}
\\
\addlinespace
\textbf{MLMC-MLCV}.
MLMC-CV with the CVs based on the surrogate models $g_0,\ldots,g_L$ on level 0 and on $h_1,\ldots,h_L$ on levels $\ell>0$. & 
$
\hat{\theta}_{0}^{(0)} - \boldsymbol{\alpha}_0^\intercal 
\begin{bmatrix}
    \hat{\tau}_0^{(0)} - \tau_0\\
    \vdots\\
    \hat{\tau}_L^{(0)} - \tau_L
\end{bmatrix}
+ \sum\limits_{\ell=1}^{L} \left(
\hat{T}_{\ell}^{(\ell)} - \boldsymbol{\alpha}_\ell^\intercal 
\begin{bmatrix}
    \hat{U}_1^{(\ell)} - U_1\\
    \vdots\\
    \hat{U}_L^{(\ell)} - U_L
\end{bmatrix}
\right)$.
&
\ref{estimatorMLCV},~\ref{correctionCV}
\\
\addlinespace
\midrule
\textbf{MLMC-CV[0]}. 
MLMC-CV using only one CV based on the surrogate $g_0$ on level 0 and no CVs on levels $\ell>0$ &
$
\hat{\theta}_{0}^{(0)} - \alpha_0(\hat{\tau}_0^{(0)} - \tau_0)
+ \sum\limits_{\ell=1}^{L} \hat{T}_{\ell}^{(\ell)}$. &
\\
\addlinespace   
\textbf{MLMC-MLCV[0]}. 
MLMC-MLCV using only CVs based on the surrogates $g_0$ and $g_1$ on level 0, and $h_1$ on levels $\ell>0$. & 
$
\hat{\theta}_{0}^{(0)}
- \boldsymbol{\alpha}_0^\intercal
\begin{bmatrix}
    \hat{\tau}_0^{(0)} - \tau_0\\
    \hat{\tau}_1^{(0)} - \tau_1
\end{bmatrix}
+ \sum\limits_{\ell=1}^{L} \left(
    \hat{T}_{\ell}^{(\ell)} - \alpha_\ell(\hat{U}_1^{(\ell)} - U_1)
\right)$. &
\\
\bottomrule
    \end{tabular}
\end{sidewaystable}

In the surrogate-based variants of MLMC, the cost of evaluating the surrogate models is assumed to be negligible compared to the costs of evaluating the simulators $f_0, \ldots, f_L$.
Therefore, the total computational cost of the MLMC-* estimator reduces to the cost of the MLMC estimator, $\cost(\hat{\theta}^{\textnormal{MLMC-MLCV}}) 
= 
\cost(\hat{\theta}^{\textnormal{MLMC}})$, given by \cref{eq:cost_MLMC}.
Similarly to the case of the MLMC estimator (see, e.g., \cite{Mycek2019}), the optimal sample sizes $(n_\ell^*)_{\ell=0}^L$ such that $\sum_{\ell=0}^L \mathcal{V}^{\textnormal{CV}}_\ell / n_\ell^*$ is minimal under a constrained computational budget of $\mathcal{C}$ are given by
\begin{equation}\label{eq:nl_opt}
    n_\ell^*
    =
    \dfrac{\mathcal{C}}{\mathcal{S}_L}
    \sqrt{\dfrac{\mathcal{V}^{\textnormal{CV}}_\ell}{\mathcal{C}_\ell + \mathcal{C}_{\ell-1}}},
    \quad
    \mbox{with }
    \mathcal{S}_\ell
    :=
    \sum_{\ell'=0}^\ell \sqrt{(\mathcal{C}_{\ell'} + \mathcal{C}_{\ell'-1}) \mathcal{V}^{\textnormal{CV}}_{\ell'}},
\end{equation}
so that $\sum_{\ell=0}^L \mathcal{V}^{\textnormal{CV}}_\ell / n_\ell^* = S_L^2 / \mathcal{C}$.
As a consequence, 
\begin{equation}\label{eq:var_bound_nl_opt_SL2}
    \V[\hat{\theta}_L^{\textnormal{MLMC-MLCV}}(\boldsymbol{\alpha}_{\ell}^* )] 
    \lesssim
    \dfrac{S_L^2}{\mathcal{C}},
\end{equation}
with an equality between the left- and right-hand sides for the expectation estimators.

In practice, $\mathcal{V}^{\textnormal{CV}}_{\ell}$ is not known and must be estimated for each level.
In this work, we consider the sequential algorithm proposed in~\cite[Algorithm~2]{Mycek2019} for the MLMC.
The algorithm starts from an initial, small number of samples $n_{\ell}^{\textnormal{init}}$ on each level. Then, it selects the optimal level on which to increase the sample size by an inflation factor $r_\ell > 1$, i.e.\ the level $\ell^*$ on which the reduction in total variance relative to the additional computational effort achieved by inflating the sample size by $r_{\ell^*}$ is maximal. 

\begin{algorithm}[!htbp]
\caption{Simplified MLMC-* algorithm inspired by~\cite{Mycek2019}.}
\begin{algorithmic}[1]
\REQUIRE $n_\ell^{\textnormal{init}} > 1$, $r_\ell > 1$, surrogate models (depending on the method), and budget $\mathcal{C}$.
\STATE Set consumed budget to $\tilde{\mathcal{C}}=0$ and $\delta n_\ell = n_\ell^{\textnormal{init}}$ samples on levels $\ell \leq L$;
\WHILE {$\tilde{\mathcal{C}} \leq \mathcal{C}$}
\STATE compute $\delta n_\ell$ samples on each level by evaluating $f_\ell$ and the appropriate surrogates;
\STATE update sample size on each level: $n_\ell \leftarrow n_\ell + \delta n_\ell$;
\STATE update consumed budget: $\tilde{\mathcal{C}} \leftarrow \tilde{\mathcal{C}} + \sum_{\ell=0}^{L} \delta n_\ell (\mathcal{C}_\ell+\mathcal{C}_{\ell-1})$;
\STATE estimate the optimal CV parameter(s) on each level;
\STATE compute/update CV estimates for $\hat{T}_\ell^{(\ell)}$ and $\mathcal{V}^{\textnormal{CV}}_\ell$ from samples on levels $\ell \leq L$;
\STATE select level 
    $\ell^* = \argmax \limits_{0 \leq \ell \leq L}  
    \dfrac{\mathcal{V}^{\textnormal{CV}}_\ell}{ r_\ell n_\ell^2(\mathcal{C}_\ell+\mathcal{C}_{\ell-1})}$;
\STATE $\delta n_{\ell^*} \leftarrow \lfloor (r_{\ell^*} - 1) n_{\ell^*} \rfloor$, \; $\delta n_{\ell \ne \ell^*} \leftarrow 0$;
\ENDWHILE
\RETURN $\hat{\theta}_L^{\textnormal{MLMC-*}}$, the MLMC-* estimate of $\theta_L$.
\end{algorithmic}
\label{algoMLCV}
\end{algorithm}


\section{Numerical experiments}\label{sec:expe}
We demonstrate the value of our MLMC-MLCV method on the uncertain heat equation problem proposed in~\cite{Geraci2015_MultifidelityControlVariate}, with seven random input variables modeling the uncertainty in the diffusion coefficient and the initial condition. 
Although academic, this problem is nonetheless challenging because of its moderate stochastic dimensionality, the non-linearity of the solution with respect to the inputs, and the high variance of the random output. 
The test problem is summarized in \cref{numModel}. 
The surrogate models used for the control variates are described in \cref{surrModel}, and the results from numerical experiments are presented and discussed in \cref{resModel}.

\subsection{Problem description}\label{numModel}
We consider the partial differential equation describing the time-evolution of the temperature $u(x,t;\vX)$ in a 1D rod of unit length over the time interval $[0,T]$, with uncertain (random) initial data $u_0$ and thermal diffusivity $\nu$,
\begin{equation}\label{eq:heateq}
\begin{cases}
\dfrac{\partial u(x,t;\vX)}{\partial t} = \nu(\vX) \dfrac{\partial^2 u(x,t;\vX)}{\partial x^2}, & x \in \mathcal{D} := (0,1), \quad t \in [0,T], \\[\bigskipamount]
u(x,0;\vX) = u_0(x;\vX), & x \in \mathcal{D}, \\[\medskipamount]
u(0,t; \vX) = u(1,t; \vX) = 0, & t \in [0,T],
\end{cases}
\end{equation}
where $\vX \colon \Omega \to \Xi$ is a random vector modelling the uncertainty in the input parameters, and where $\nu(\vX) > 0$ almost surely.
The solution of \cref{eq:heateq} may be expressed as
\begin{equation}\label{eq:heateq_sol_fourier}
u(x,t;\vX)
= \sum_{k=1}^{\infty}
a_k(\vX) \exp(-\nu(\vX) k^2\pi^2 t) \sin(k\pi x)
\end{equation}
with
\begin{equation}\label{eq:heateq_sol_fourier_ak}
a_k(\vX) = 2 \int_{\mathcal{D}} u_0(x; \vX) \sin(k\pi x) \, \mathrm{d}x .
\end{equation}
The initial condition is chosen to have the same prescribed form as in~\cite{Geraci2015_MultifidelityControlVariate}. Specifically, we consider
$u_0(x; \vX) 
= \mathcal{G}(\vX) \mathcal{F}_1(x) 
+ \mathcal{I}(\vX) \mathcal{F}_2(x)$
with
\begin{align}%
\mathcal{F}_1(x) &= \sin(\pi x),\\
\mathcal{F}_2(x) &= \sin(2 \pi x)+\sin(3 \pi x)+50(\sin(9 \pi x) + \sin(21 \pi x)),\\
\mathcal{I}(\vX) &= \frac{7}{2} \left[\sin(X_1)+7\sin(X_2)^2+0.1X_3^4\sin(X_1) \right],\\
\mathcal{G}(\vX) &= 50 (4|X_5| - 1)(4|X_6| - 1)(4|X_7| - 1),
\end{align}
which allows to control the spectral content of the solution $u$.
Furthermore, as in~\cite{Geraci2015_MultifidelityControlVariate}, the diffusion coefficient is modelled by $\nu(\vX) = X_4$.
The random output variable of interest is defined as the integral of the temperature along the rod at final time $T$,
\begin{align}
    \mathcal{M}(\vX) &= \int_{\mathcal{D}} u(x,T; \vX) \, \mathrm{d}x\\
    &= \sum_{k=1}^{\infty}
    a_k(\vX)
    \int_{\mathcal{D}} \exp(-\nu(\vX) k^2\pi^2 T) \sin(k\pi x) \, \mathrm{d}x 
    \label{eq:heateq_Q}\\
    &= 
    \mathcal{G}(\vX) \mathcal{H}_1(\vX)
    +
    \mathcal{I}(\vX) \left[
    \mathcal{H}_3(\vX)
    +
    50 \mathcal{H}_9(\vX)
    +
    50 \mathcal{H}_{21}(\vX)
    \right],
\end{align}
where $\mathcal{H}_k(\vX) = \frac{2}{k\pi} \exp(-\nu(\vX) k^2\pi^2 T)$.
In this experiment, we seek to estimate the expectation $\E[\mathcal{M}(\vX)]$, for a given uncertain setting.
Consistently with~\cite{Geraci2015_MultifidelityControlVariate}, we consider the random variables $X_1,\ldots,X_7$ to be independent and distributed as
\begin{align}\label{eq:distrib}
    X_1, X_2, X_3 &\sim \mathcal{U}[-\pi,\pi],
    & X_4 &\sim \mathcal{U}[\nu_{\textnormal{min}}, \nu_{\textnormal{max}}], 
    & X_5, X_6, X_7 &\sim \mathcal{U}[-1,1].
\end{align}
The expected value $\E[\mathcal{M}(\vX)]$ is then given by
\begin{equation}\label{eq:mu_exact}
    \E[\mathcal{M}(\vX)]
    = 50 H_1 + \frac{49}{4}\left(H_3 + 50 H_9 + 50 H_{21} \right),
\end{equation}
where 
\begin{equation}
    H_k
    = \E[\mathcal{H}_k(\vX)]
    = 
    \frac{2}{k^3 \pi^3 T}
    \frac{\exp(-\nu_{\textnormal{min}}k^2\pi^2 T) - \exp(-\nu_{\textnormal{max}}k^2\pi^2 T)}
    {\nu_{\textnormal{max}} - \nu_{\textnormal{min}}}.
\end{equation}
Finally, we set $T=0.5$, $\nu_{\textnormal{min}} = 0.001$ and $\nu_{\textnormal{max}} = 0.009$, resulting in $\E[\mathcal{M}(\vX)] \approx 41.98$.

Numerically, $\mathcal{M}(\vX)$ is approximated by truncating the Fourier expansion in \cref{eq:heateq_Q} to ${K < \infty}$ modes and by approximating the integrals in \cref{eq:heateq_Q,eq:heateq_sol_fourier_ak} by a trapezoidal quadrature rule with equispaced nodes in $[0,1]$.
The multilevel hierarchy of simulators $\{f_\ell\}_{\ell=0}^L$ is then defined according to the number of quadrature nodes $N_\ell$ used for the approximation at level $\ell$.
Specifically, $\mathcal{M}(\vX)$ is approximated at level $\ell$ by 
\begin{equation}\label{eq:heateq_Y_discr}
    Y_\ell = f_\ell(\vX)
    = \sum_{k=1}^{K} A_k^\ell(\vX) B_k^\ell(\vX),
\end{equation}
with 
\begin{equation}
    A_k^\ell(\vX)
    =
    2 \sum_{i=1}^{N_\ell} w_i u_0(x_i; \vX) \sin(k\pi x_i),
    \;
    B_k^\ell(\vX)
    = \exp(-\nu(\vX) k^2\pi^2 T) \sum_{i=1}^{N_\ell} w_i \sin(k\pi x_i),
\end{equation}
where $\{(x_i, w_i)\}_{i=1}^{N_\ell}$ are the pairs of quadrature nodes and associated weights on level $\ell$. 
It is then natural to assume that the computational cost $\mathcal{C}_\ell$ of an evaluation of $f_\ell$ is $\mathcal{O}(KN_\ell)$.
The statistic of interest is thus $\theta = \theta_L = \E[f_L(\vX)]$, whose MC estimator 
$\hat{\theta}_L
=
n_L^{-1}\sum_{i=1}^{n_L} f_L(\vX^{(L,i)})$
will represent the baseline estimator for our experiments.
Besides, the quality of all the presented estimators will be assessed in terms of their root mean square error (RMSE) w.r.t.\ the exact statistic $\E[\mathcal{M}(\vX)]$ given by \cref{eq:mu_exact}.

In the following experiments, we set the number of quadrature nodes $K=21$ and the number of levels to 4 (i.e.\ $L=3$).
Furthermore, we set $N_\ell=120\times2^{L-\ell}$ so that evaluating $f_\ell$ is twice as expensive as evaluating $f_{\ell-1}$.
Table~\ref{QuadratureLevels} summarizes the number of quadrature nodes and the evaluation cost per level.
Note that the costs are normalized so that $\mathcal{C}_3 = 1$.

\begin{table}[!ht]
\centering
\caption{Number of quadrature nodes $N_\ell$, simulator evaluation cost $\mathcal{C}_\ell$ and MLMC correction evaluation cost $\mathcal{C}_\ell+\mathcal{C}_{\ell-1}$ per level.}
\begin{tabular}[t]{ccccc}
\toprule
$\ell$ & 0 & 1 & 2 & 3\\
\midrule
$N_\ell$ & 15 & 30 & 60 & 120 \\
$\mathcal{C}_\ell$ & 0.125 & 0.25 & 0.5 & 1 \\
$\mathcal{C}_\ell + \mathcal{C}_{\ell-1}$ & 0.125 & 0.375 & 0.75 & 1.5 \\
\bottomrule
\end{tabular}\label{QuadratureLevels}
\end{table}%

\subsection{Surrogate models}\label{surrModel}
We will mostly use PC models (see \cref{sec:pce}) for the surrogate-based CV estimators. 
Constructing high-quality surrogates in 7 dimensions can be hard, especially in the presence of non-linearities and with a limited sample size. 
To avoid overfitting, we resort to the \emph{least angle regression} (LARS) procedure~\cite{Efron2004_LeastAngleRegression}, which is a model-selection regression method that promotes sparsity.
More precisely, we employ the basis-adaptive hybrid LARS algorithm proposed by~\cite[Fig.~5]{Blatman2011_AdaptiveSparsePolynomial} for the selection of sparse PC bases.
For a given design of experiment (DoE) for the PC surrogate construction, this algorithm applies the LARS procedure on candidate PC bases $\mathcal{A}_p$ of increasing total polynomial degree $p = 1, \ldots,p_{\textnormal{max}}$, resulting for each $p$ in the selection of a limited number $|\tilde{\mathcal{A}}_p| < |\mathcal{A}_p|$ of basis polynomial functions.
For each $p$, a PC surrogate is constructed by classical least-squares regression on the corresponding reduced (or active) PC basis $\tilde{\mathcal{A}}_{p}$, and its quality is estimated using a corrected leave-one-out cross-validation procedure~\cite{Blatman2011_AdaptiveSparsePolynomial}.
The best surrogate according to this quality measure is eventually retained, and the associated reduced PC basis is denoted by $\tilde{\mathcal{A}}_{p^*}$.

In our experiments, we set $p_{\textnormal{max}} = 16$ and the construction budget to 400 times the evaluation cost of $f_3$, i.e.\ $\mathcal{C}^{\textnormal{DoE}} = 400$. 
This budget is distributed equally among the different levels, so that the associated evaluation cost on each level corresponds to the cost of 100 $f_3$ evaluations, i.e.\ $n_\ell^{\textnormal{DoE}} \mathcal{C}_\ell = 100$, as reported in \cref{tab:pces,tab:pces_ml}. 
Once constructed, the quality of a surrogate $g$ of $f$ is assessed in terms of the $Q^2$ measure,
\begin{equation}
    Q^2(g, f)=1-\frac{\E[(g(\vX)-f(\vX))^2]}{\V[f(\vX)]},
\end{equation}
estimated using a test sample of size $n_{\textnormal{test}}=\num{10000}$, which is more robust than the corrected leave-one-out measure used for the model selection.
The higher the $Q^2$ value, the higher-quality the associated surrogate.

For the MLCV estimator, we learn a PC model $g_\ell$ of $f_\ell$ from a training DoE of size $n_\ell^{\textnormal{DoE}}$ generated by latin hypercube sampling (LHS) improved by simulated annealing. 
The training DoE samples for the surrogate construction are different for each level and independent. 
\Cref{tab:pces} summarizes the properties of the different PC models.
Except for $g_3$, which is built with only 100 points, all the PC models have a good $Q^2$, greater than 0.8. 
We observe that the basis-adaptive LARS algorithm has selected a decreasing polynomial degree $p^*$ with level $\ell$, while retaining only a limited number $|\tilde{\mathcal{A}}_{p^*}|$ of polynomials in these reduced bases.
Thus, although $g_1$ and $g_2$ have distinct degrees, the sizes of the associated reduced bases are similar.

\begin{table}[!htbp]
\caption{PC models for CV and MLCV estimators, with their sample size, degree and quality measure. These models are built with the basis-adaptive LARS algorithm of~\cite[Fig.~5]{Blatman2011_AdaptiveSparsePolynomial}.}
\label{tab:pces}
\centering
\begin{tabular}[c]{ccccc}
\toprule
PC models & $g_0$ & $g_1$ & $g_2$ & $g_3$\\
\midrule
$n_\ell^{\textnormal{DoE}}$ & 800 & 400 & 200 & 100 \\
$p^*$ & 10 & 8 & 6 & 4 \\
$|\tilde{\mathcal{A}}_{p^*}|$ & 211 & 77 & 74 & 14 \\
$Q^2$ & 0.98 & 0.94 & 0.86 & 0.59\\
\bottomrule
\end{tabular}
\end{table}

For the MLMC-based estimators, we learn a PC model $g_\ell$ of $f_\ell$ for $\ell=0,\ldots,L$ and a PC model $h_\ell$ of $f_\ell-f_{\ell-1}$ for $\ell = 1,\ldots,L$. 
In practice, training points used for the construction of $g_0, \ldots, g_L$ may be reused for the construction of $h_1, \ldots, h_L$ using nested DoEs.
However, it should be noted that the generation of nested LHS DoEs is not as straightforward as for purely random DoEs.
While using nested random DoEs is a perfectly valid strategy, we opt for an alternative choice based on LHS.
Specifically, we first generate a DoE $\mathcal{X}^{(0)}_{\textnormal{DoE}}$ of size $n_0^{\textnormal{DoE}}$ using LHS improved by simulated annealing.
Then, for $\ell=1, \ldots, L$, we sequentially extract a DoE $\mathcal{X}^{(\ell)}_{\textnormal{DoE}} \subset \mathcal{X}^{(\ell-1)}_{\textnormal{DoE}}$ of size $n_\ell^{\textnormal{DoE}}$.
The subset $\mathcal{X}^{(\ell)}_{\textnormal{DoE}}$ is selected such that it has minimal centered $L^2$ discrepancy among a pool of 1 million random candidate subsets.
\Cref{tab:pces_ml} summarizes the properties of these different PC models. The PC models again have a quality measure higher than 0.8, except for $g_3$ and $h_3$, which were constructed from only 100 evaluations. 
Note that the values of $p^*$, $|\tilde{\mathcal{A}}_{p^*}|$ and $Q^2$ are of the same order of magnitude as those of the PC surrogate models reported in \cref{tab:pces}.

\begin{table}[!htbp]
\caption{PC models for the estimators combining MLMC and CV techniques, with their sample size, degree and quality measure. More precisely, $g_\ell$ is the surrogate of $f_\ell$, $h_\ell$ is the surrogate of $f_\ell-f_{\ell-1}$ and the samples used to train $h_\ell$ are also used to train $g_\ell$ and $g_{\ell-1}$. These models are built with the basis-adaptive LARS algorithm of~\cite[Fig.~5]{Blatman2011_AdaptiveSparsePolynomial}.}
\label{tab:pces_ml}
\centering
\begin{tabular}[c]{cccccccc}
\toprule
PC models & $g_0$ & $g_1$ & $g_2$ & $g_3$ & $h_1$ & $h_2$ & $h_3$\\
\midrule
$n_\ell^{\textnormal{DoE}}$ & 800 & 400 & 200 & 100 & 400 & 200 & 100\\
$p^*$ & 10 & 8 & 8 & 2 & 9 & 7 & 5 \\
$|\tilde{\mathcal{A}}_{p^*}|$ & 211 & 83 & 57 & 6 & 112 & 71 & 30 \\
$Q^2$ & 0.98 & 0.95 & 0.86 & 0.21 & 0.99 & 0.82 & 0.63 \\
\bottomrule
\end{tabular}
\end{table}

Hereafter, the CV and MLCV estimators use the PC models of \cref{tab:pces}, while the MLMC-based estimators use the PC models presented in \cref{tab:pces_ml}.

\subsection{Results}\label{resModel}

First, in \cref{res:cv}, we illustrate the use of one or several control variates to reduce the variance of a single-level MC estimator. 
Then, we compare the MLCV and MLMC-MLCV approaches with the MC and MLMC estimators in \cref{res:mlmcmcv}, and we discuss variants of the MLMC-CV and MLMC-MLCV approaches, considering only a limited subset of the surrogate models, in \cref{res:mlmcmcv_vs_mlmf}. 
We conclude the analysis by reporting the estimation budget allocation across levels resulting from the various MLMC-based methods in \cref{res:budget_alloc}.
%
In practice, all the MLMC-based estimators are built using \cref{algoMLCV}. 
Unless stated otherwise, the parameters are set to $n_{\ell}^{\textnormal{init}}=30$ and $r_\ell = 1.1$ for $\ell=0,\ldots,3$.
The quality of the various estimators will be assessed in terms of their RMSE w.r.t.\ $\E[\mathcal{M}(\vX)]$, estimated from 500 repetitions of the experiment.

\subsubsection{Single-level MC and control variates}\label{res:cv}

In this first part, we consider only the finest simulator $f_3$ and try to reduce the variance of the MC estimator of $\E[f_3(\vX)]$ by means of surrogate-based CVs. 
For that purpose, we consider a first-order Taylor polynomial expansion $g^{\text{T}_1}_{3}$ around $\boldsymbol{\mu}_\vX = \mathbb{E}[\vX]$ (see \cref{sec:taylor,app:taylor}) and the PC model $g^{\text{PC}}_{3}$ described in \cref{tab:pces}. 

\begin{table}[ht]
\caption{Pearson correlation coefficients between the finest simulator $f_3$ and the corresponding PC model (see \cref{tab:pces}) and $T_1$ model around $\boldsymbol{\mu}_\vX = \mathbb{E}[\vX]$, estimated with a sample of size $\num{1000}$.}
\label{tab:cv_pearson}
\renewcommand{\arraystretch}{1.2}
\centering
\begin{tabular}[t]{@{\;}c|ccc@{\;}}
& $Y_3$ & $g^{\textnormal{PC}}_{3}(\vX)$ & $g^{\text{T}_1}_{3}(\vX)$ \\
\hline
$Y_3$ & 1.00 & 0.80 & 0.57 \\
$g^{\textnormal{PC}}_{3}(\vX)$ & 0.80 & 1.00 & 0.64 \\
$g^{\textnormal{T}_1}_{3}(\vX)$ & 0.57 & 0.64 & 1.00 \\
\hline
\end{tabular}%
\end{table}

\Cref{tab:cv_pearson} shows that $g^{\text{PC}}_{3}(\vX)$ is well correlated with $Y_3=f_3(\vX)$, with a Pearson coefficient of 0.8, while $g^{\text{T}_1}_{3}$ is less so, with a coefficient of 0.57, which may be explained by the strong non-linearity of $f_3$.
According to \cref{VRMCV}, these correlation coefficients lead to a theoretical variance reduction factor $R^2$ of 64\% when using $g^{\text{PC}}_{3}(\vX)$ as a single CV, and of about 32\% when using $g^{\text{T}_1}_{3}(\vX)$ as a single CV.
Using both CVs results in a reduction factor of 65\%. 
This minor increase in $R^2$ can be explained by the modest correlation coefficient of 0.64 between $g^{\text{T}_1}_{3}(\vX)$ and $g^{\text{PC}}_{3}(\vX)$.
Note that a reduction factor of $R^2$ in the variance corresponds to a reduction factor of $1-\sqrt{1-R^2}$ in the standard deviation.

\begin{figure}[ht]
\centering
\begin{subfigure}[t]{0.49\textwidth}
\centering
\includegraphics[width=\textwidth]{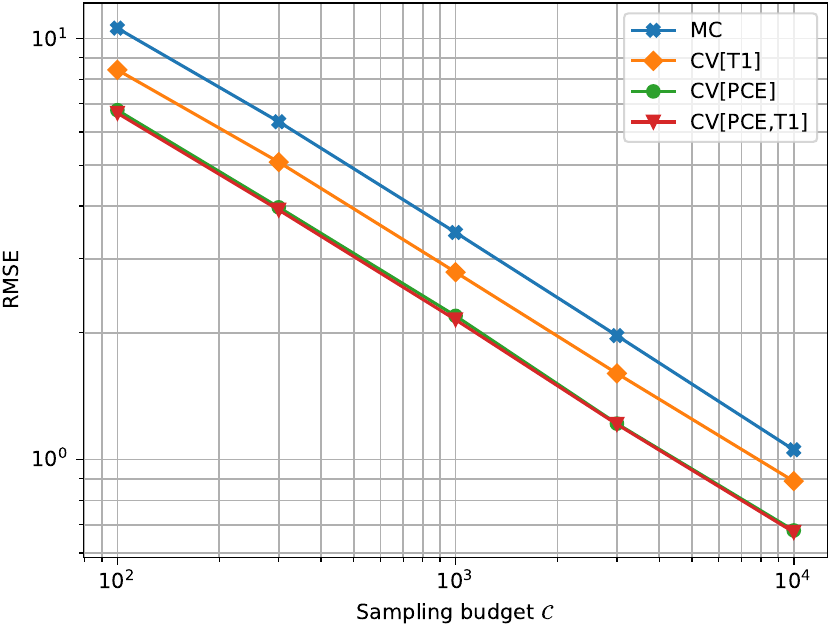}
\caption{The evaluation cost does not include the cost of constructing the surrogate models.}
\label{fig:cv_without_offset}
\end{subfigure}\hfill%
\begin{subfigure}[t]{0.49\textwidth}
\centering
\includegraphics[width=\textwidth]{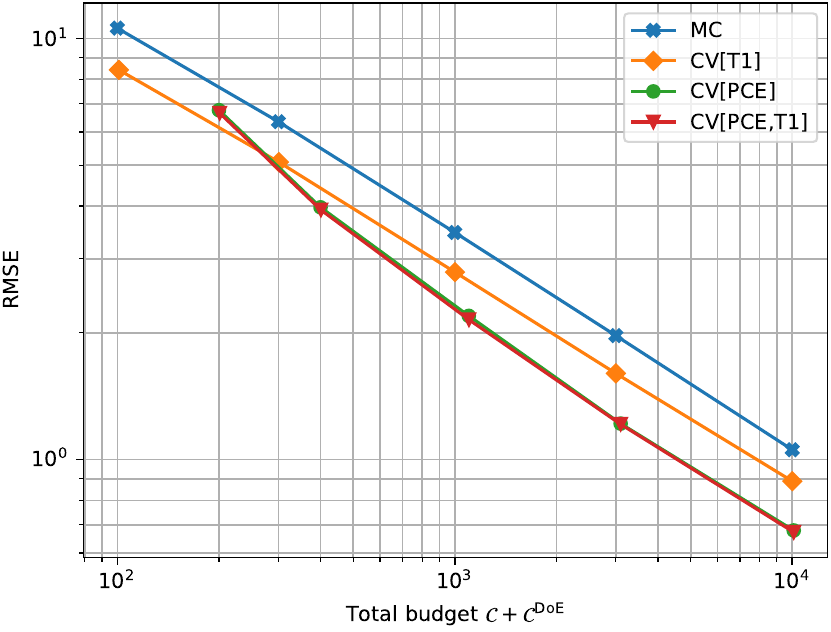}
\caption{The evaluation cost includes the cost of constructing the surrogate models.}
\label{fig:cv_with_offset} 
\end{subfigure}
\caption{RMSE of the MC and CV estimators of $\theta=\E[f_3(\vX)]$ (see \cref{eq:heateq_Y_discr}) with respect to the number $\mathcal{C}\in\{\num{100};\num{300};\num{1000};\num{3000};\num{10000}\}$ of $f_3$ evaluations. The CV uses either a first-order Taylor approximation (T1) of $f_3$ or a PC approximation of $f_3$ trained from $100$ evaluations ($\mathcal{C}^{\textnormal{DoE}} = 100$), or both.
The RMSEs are computed using 500 replicates.}
\label{fig:cv_with_or_without_offset}
\end{figure}

These theoretical expectations are reflected in \cref{fig:cv_without_offset} with an RMSE reduction of about 20\% when using $g^{\text{T}_1}_{3}$ and 40\% when using  $g^{\text{PC}}_{3}$ alone or jointly with $g^{\text{T}_1}_{3}$. 
This figure confirms that $g^{\text{PC}}_{3}$ provides a better CV than $g^{\text{T}_1}_{3}$, reducing the RMSE of the MC estimator twice as much, regardless of the computational budget. 
However, the construction cost of the surrogate is not the same.
While constructing $g^{\text{T}_1}_{3}$ requires only one evaluation of $f_3$ and its Jacobian matrix, namely at $\boldsymbol{\mu}_\vX$, the construction of $g^{\text{PC}}_{3}$ involved 100 $f_3$ evaluations. 
In the case where the surrogate is built specifically for the estimation of the statistic, the real estimation cost is higher as it includes this construction cost. \Cref{fig:cv_with_offset} illustrates this difference by including the surrogate construction cost in the total evaluation cost. 
As a result, a significant offset appears when using $g^{\text{PC}}_{3}$, since part of the computational budget (namely 100) is used for the surrogate construction and thus not for the estimation. 
Specifically, the RMSEs of the CV estimators using $g^{\text{PC}}_{3}$ get below that of the CV estimator using only $g^{\text{T}_1}_{3}$ for a budget of 300 $f_3$ evaluations.
For budgets under 200, using only $g^{\text{T}_1}_{3}$ is preferable, even without the analytical gradient available, which would then incur a construction cost of 8 $f_3$ evaluations to approximate the gradient using finite differences.
This figure also illustrates \cref{prop:add_cv}, that is, increasing the number of control variates improves (rigorously speaking, does not deteriorate) the variance of the CV estimator.

\subsubsection{MLCV and MLMC-MLCV}\label{res:mlmcmcv}

\Cref{fig:mc_vs_ml} compares the MLCV and MLMC-MLCV estimators proposed in \cref{eq:MLCV,estimatorMLCV} with the classical MC and MLMC estimators. 
This comparison is repeated for different budgets $\mathcal{C}$, expressed in terms of the equivalent number of $f_3$ evaluations.
Note that the MC and MLCV estimators only use evaluations of the finest simulator $f_3$, so that $\mathcal{C} = n_3$, while MLMC-based estimators use all the simulators, so that $\mathcal{C}$ is given in terms of the MLMC cost \cref{eq:cost_MLMC}, indeed corresponding to the equivalent number of $f_3$ evaluations, since the costs are normalized such that $\mathcal{C}_3 = 1$. 
From here on, only PC-based surrogates will be used, so that we omit the superscript ``PC'' in the notations of the surrogate models.

\begin{figure}[!ht]
\centering
\begin{subfigure}[t]{0.49\textwidth}
\centering
\includegraphics[width=\textwidth]{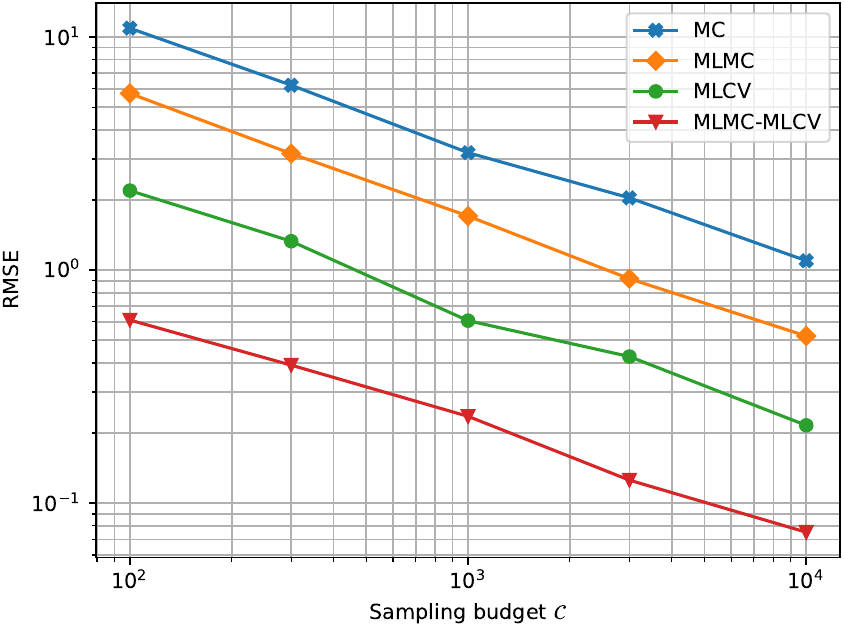}
\caption{The evaluation cost does not include the cost of construction of the surrogate models.}
\label{fig:mc_vs_ml_without_offset}
\end{subfigure}\hfill%
\begin{subfigure}[t]{0.49\textwidth}
\centering
\includegraphics[width=\textwidth]{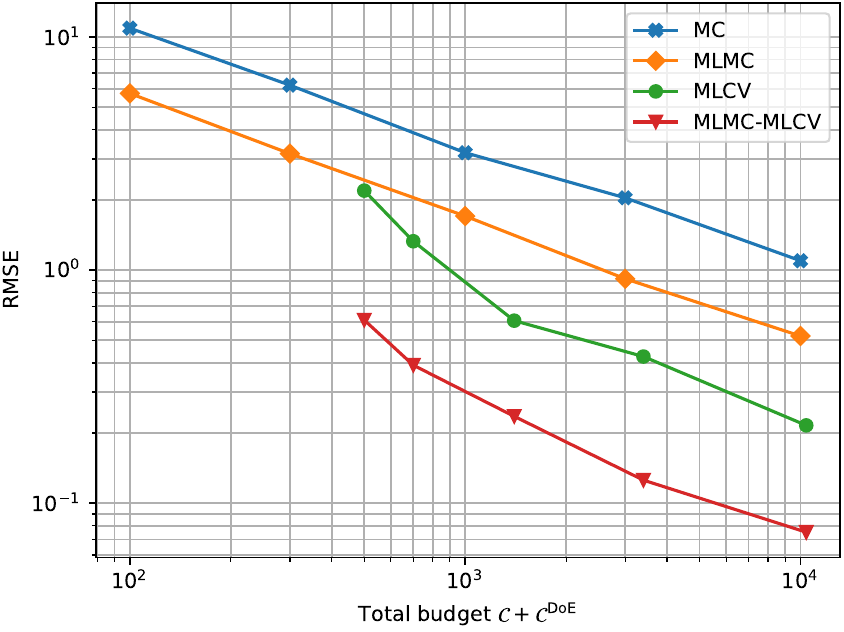}
\caption{The evaluation cost includes the cost of construction of the surrogate models.}
\label{fig:mc_vs_ml_with_offset} 
\end{subfigure}
\caption{RMSE of the MC, MLMC, MLCV and MLMC-MLCV estimators of $\theta=\E[f_3(\vX)]$ (see \cref{eq:heateq_Y_discr}) with respect to the sampling budget $\mathcal{C}\in\{\num{100};\num{300};\num{1000};\num{3000};\num{10000}\}$ of $f_3$ evaluations. 
The surrogates used for the MLCV and MLMC-MLCV estimators are described in \cref{tab:pces,tab:pces_ml}, respectively, and their total construction cost is $\mathcal{C}^{\textnormal{DoE}}=400$.
The RMSEs are estimated using 500 replicates.}
\label{fig:mc_vs_ml}
\end{figure}

\paragraph{MLMC vs.\ MC}
\Cref{fig:mc_vs_ml_without_offset} shows that, for a given sampling budget $\mathcal{C}$, the MC estimator is the least accurate. This can be explained by the fact that it only has access to the finest simulator, $f_3$, whose cost only allows  a limited number of evaluations. 
On the contrary, the MLMC estimator spreads this sampling budget over the four simulators.
Ideally, the optimal sample allocation of MLMC \cref{eq:nl_opt} results in many coarse, cheap evaluations and few fine, expensive evaluations.
This is typically the case when the outputs of the simulators are highly correlated and the associated computational cost grows exponentially.
\Cref{tab:mlmc_pearson} shows that the first assumption holds.
The second assumption also holds since $\mathcal{C}_\ell = \mathcal{O}(K N_\ell) = \mathcal{O}(N_\ell)$, $K$ being fixed, i.e., the evaluation cost grows linearly with the number of quadrature nodes.
As a consequence, the MLMC estimator has a lower RMSE than the standard MC estimator.

\begin{table}[!ht]
\caption{Pearson correlation coefficients between the finest output $Y_3=f_3(\vX)$ and the corresponding control variates based on the PC models of \cref{tab:pces} used for the MLCV method, estimated with a sample of size $\num{1000}$.}
\label{tab:mlcv_pearson}
\renewcommand{\arraystretch}{1.2}
\centering
\begin{tabular}[t]{@{\;}c|ccccc@{\;}}
& $Y_3$ & $g_0(\vX)$ & $g_1(\vX)$ & $g_2(\vX)$ & $g_3(\vX)$ \\
\hline
$Y_3$      & 1.00 & 0.96 & 0.97 & 0.93 & 0.80\\
$g_0(\vX)$ & 0.96 & 1.00 & 0.96 & 0.92 & 0.80\\
$g_1(\vX)$ & 0.97 & 0.96 & 1.00 & 0.95 & 0.81\\
$g_2(\vX)$ & 0.93 & 0.92 & 0.95 & 1.00 & 0.80\\
$g_3(\vX)$ & 0.80 & 0.96 & 0.81 & 0.80 & 1.00\\
\hline
\end{tabular}%
\end{table}

\paragraph{MLCV vs.\ MLMC}
For this experiment, the MLCV estimator is more accurate than the MLMC estimator. 
Based on one control variate per level $\ell$, based on the PC model $g_\ell$ of $f_\ell$ from \cref{tab:pces}, this estimator dedicates all the sampling budget to the finest simulator $f_3$, and uses these control variates based on $g_0$, $g_1$, $g_2$ and $g_3$ to reduce the variance of the MC estimator at no extra cost, as the evaluation cost of a PC model $g_\ell$ is negligible compared to the evaluation cost of $f_3$. 
This MLCV technique works particularly well in this case because the control variates are highly correlated to $f_3$. 
Indeed, \cref{tab:mlcv_pearson} shows that their Pearson coefficients are at least 0.8, which guarantees a theoretical reduction of at least 94\% in the variance of the MC estimator (i.e.\ a reduction of at least about 76\% in standard deviation), corresponding to the variance reduction when using a single control variate based on $g_3$.
In fact, the variance reduction factor $R^2$ when using all the surrogates is only slightly higher, namely $R^2 \approx 95\%$ corresponding to a standard reduction factor of around 78\%, which is reflected in \cref{fig:mc_vs_ml_without_offset}.

\begin{table}[!ht]
\caption{Pearson correlation coefficients between the outputs $Y_\ell = f_\ell(\vX)$ of the simulators and $g_\ell(\vX)$ of the PC models, as well as the successive differences $Y_\ell - Y_{\ell-1}$ and the outputs $h_\ell(\vX)$ of the associated PC models, for the MLMC-based estimators. The coefficients are estimated from a sample of size $\num{1000}$.}
\label{tab:mlmc_pearson}
\renewcommand{\arraystretch}{1.2}
\centering
\resizebox{\linewidth}{!}{%
\begin{tabular}[t]{c|cccccccccccccc}
& $Y_0$ & $Y_1$ & $Y_2$ & $Y_3$ & $Y_1-Y_0$ & $Y_2-Y_1$ & $Y_3-Y_2$ & $g_0(\vX)$ & $g_1(\vX)$ & $g_2(\vX)$ & $g_3(\vX)$ & $h_1(\vX)$ & $h_2(\vX)$ & $h_3(\vX)$ \\
\hline
 $Y_0$ & \graycell 1.00 & \graycell 0.97 & \graycell 0.97 & \graycell 0.97 & -0.11 & -0.04 & -0.04 & \graycell 0.99 & \graycell 0.96 & \graycell 0.93 & \graycell 0.47 & -0.12 & -0.13 & -0.10 \\
 $Y_1$ & \graycell 0.97 & \graycell 1.00 & \graycell 1.00 & \graycell 1.00 & 0.13 & 0.19 & 0.19 & \graycell 0.97 & \graycell 0.97 & \graycell 0.93 & \graycell 0.48 & 0.12 & 0.09 & 0.09 \\
 $Y_2$ & \graycell 0.97 & \graycell 1.00 & \graycell 1.00 & \graycell 1.00 & 0.14 & 0.20 & 0.20 & \graycell 0.96 & \graycell 0.97 & \graycell 0.92 & \graycell 0.48 & 0.13 & 0.10 & 0.10 \\
 $Y_3$ & \graycell 0.97 & \graycell 1.00 & \graycell 1.00 & \graycell 1.00 & 0.14 & 0.20 & 0.20 & \graycell 0.96 & \graycell 0.97 & \graycell 0.92 & \graycell 0.48 & 0.13 & 0.10 & 0.10 \\
 
 $Y_1-Y_0$ & -0.11 & 0.13 & 0.14 & 0.14 & \graycell 1.00 & \graycell 0.96 & \graycell 0.96 & -0.10 & 0.07 & 0.01 & 0.03 & \graycell 0.99 & \graycell 0.90 & \graycell 0.80 \\
 $Y_2-Y_1$ & -0.04 & 0.19 & 0.20 & 0.20 & \graycell 0.96 & \graycell 1.00 & \graycell 1.00 & -0.03 & 0.12 & 0.06 & 0.06 & \graycell 0.94 & \graycell 0.90 & \graycell 0.81 \\
 $Y_3-Y_2$ & -0.04 & 0.19 & 0.20 & 0.20 & \graycell 0.96 & \graycell 1.00 & \graycell 1.00 & -0.03 & 0.12 & 0.06 & 0.06 & \graycell 0.95 & \graycell 0.90 & \graycell 0.81 \\

 $g_0(\vX)$ & \graycell 0.99 & \graycell 0.97 & \graycell 0.96 & \graycell 0.96 & -0.10 & -0.03 & -0.03 & \graycell 1.00 & \graycell 0.96 & \graycell 0.93 & \graycell 0.48 & -0.11 & -0.12 & -0.09 \\
 $g_1(\vX)$ & \graycell 0.96 & \graycell 0.97 & \graycell 0.97 & \graycell 0.97 & 0.07 & 0.12 & 0.12 & \graycell 0.96 & \graycell 1.00 & \graycell 0.94 & \graycell 0.47 & 0.06 & 0.04 & 0.04 \\
 $g_2(\vX)$ & \graycell 0.93 & \graycell 0.93 & \graycell 0.92 & \graycell 0.92 & 0.01 & 0.06 & 0.06 & \graycell 0.93 & \graycell 0.94 & \graycell 1.00 & \graycell 0.47 & -0.00 & -0.03 & -0.02 \\
 $g_3(\vX)$ & \graycell 0.47 & \graycell 0.48 & \graycell 0.48 & \graycell 0.48 & 0.03 & 0.06 & 0.06 & \graycell 0.48 & \graycell 0.47 & \graycell 0.47 & \graycell 1.00 & 0.03 & 0.01 & 0.02 \\

 $h_1(\vX)$ & -0.12 & 0.12 & 0.13 & 0.13 & \graycell 0.99 & \graycell 0.94 & \graycell 0.95 & -0.11 & 0.06 & -0.00 & 0.03 & \graycell 1.00 & \graycell 0.91 & \graycell 0.80 \\
 $h_2(\vX)$ & -0.13 & 0.09 & 0.10 & 0.10 & \graycell 0.90 & \graycell 0.90 & \graycell 0.90 & -0.12 & 0.04 & -0.03 & 0.01 & \graycell 0.91 & \graycell 1.00 & \graycell 0.80 \\
 $h_3(\vX)$ & -0.10 & 0.09 & 0.10 & 0.10 & \graycell 0.80 & \graycell 0.81 & \graycell 0.81 & -0.09 & 0.04 & -0.02 & 0.02 & \graycell 0.80 & \graycell 0.80 & \graycell 1.00 \\
\hline
\end{tabular}%
}
\end{table}

\paragraph{MLMC-MLCV vs.\ MLCV}
Combining the MLMC and MLCV techniques allows the resulting MLMC-MLCV estimator to reduce the variance even more significantly.
This can be explained by the very high correlation between $Y_0$, $Y_1$, $Y_2$ and $Y_3$ on the one hand, which ensures the good performance of the MLMC approach, and by the strong correlation between the control variates on the other hand, ensuring their good performance in combination with the MLMC technique.
In particular, \cref{tab:mlmc_pearson} shows that $g_0(\vX)$, $g_1(\vX)$ and $g_2(\vX)$ are highly correlated with $Y_3$, with Pearson correlation  coefficients greater than 0.9, while $g_3(\vX)$ is poorly correlated with $Y_3$, with a correlation coefficient of 0.48.
Besides, $h_1(\vX)$, $h_2(\vX)$ and $h_3(\vX)$ are well-correlated with $Y_1-Y_0$, $Y_2-Y_1$ and $Y_3-Y_2$, respectively, with correlation coefficients greater than 0.8.

\paragraph{A priori estimation of the variance reduction factor and sample allocation}
Further insights regarding the expected variance reduction of the MLMC-* estimators can be drawn from \cref{tab:variance_reduction2}, which reports the variance reduction factor $R^2_\ell$ w.r.t.\ pure MLMC on each level defined in \cref{eq:var_MLMC_MCV,eq:R2_MLMC-MLCV}, as well as the quantity $\mathcal{S}_\ell^2$ defined in \cref{eq:nl_opt}.
In addition, the variance of the MLMC-* estimators of the expectation with optimal sample allocation is $S_L^2/\mathcal{C}$ (see \cref{eq:var_bound_nl_opt_SL2}), so that the variance reduction factor of an MLMC-* estimator compared to the high-fidelity MC estimator is $1-\mathcal{S}_L^2 / (\mathcal{C}_L \V[Y_L])$, which is reported in the last rows and last column of the table.
Consequently, the variance of the different MLMC-based estimators per unit cost can be compared directly through the ratio of their respective $S_L^2$.
For instance, we observe that the variance of the MLMC-MLCV estimator is significantly reduced compared to that of the MLMC estimator, by about $1- 40.04/2797.65 \approx98.6\%$, resulting in a reduction in standard deviation of about 88\%.
Again, this is well reflected in \cref{fig:mc_vs_ml_without_offset}. 
\Cref{tab:variance_reduction2} also reports the share of computational work across correction levels, namely $n_\ell^* (\mathcal{C}_\ell + \mathcal{C}_{\ell-1}) / \mathcal{C}$, with $n_\ell^*$ as defined in \cref{eq:nl_opt}, for the correction on level $\ell$.
These anticipated shares are consistent with those obtained by using \cref{algoMLCV}, reported in \cref{fig:allocation_cost}.
It is important to note that the 
bottom rows of
\cref{tab:variance_reduction2} are deduced from estimates of $\V[Y_\ell]$, $\mathcal{V}_\ell$ and $R^2_\ell$, reported in the top rows. 
This means that, in practice, these estimates can be obtained from a few evaluations of $f_0,\ldots,f_L$ (and of the appropriate surrogate models), to 
anticipate
the expected variance reduction factor of the MLMC-* estimators and apply \cref{algoMLCV} to the most promising.

\begin{remark}
\Cref{tab:variance_reduction2} highlights the fast decay of $\mathcal{V}_{\ell}$ with $\ell$, which is a favorable scenario for MLMC.
In fact, \cref{tab:variance_reduction2} reports that the variance of the plain MLMC estimator is reduced by about 75\% compared to that of the crude, high-fidelity MC estimator, i.e., a reduction by a factor 4.
Equivalently, this means that a 4-fold increase of the computational budget is needed to achieve the same variance with the high-fidelity MC estimator as that of the plain MLMC.
This favorable setting for MLMC makes it all the more challenging to futher reduce the variance.
Yet, we observe that the proposed MLMC-* estimators are able to achieve a much smaller variance, especially by reducing the MLMC variance on the coarsest levels.
\end{remark}

\begin{table}[!ht]\centering
\caption{Relevant quantities for the MLMC-based estimators, with $n_\ell^*$ and $\mathcal{S}^2_\ell$ as defined in \cref{eq:nl_opt}, $\mathcal{C}_\ell$ as reported in \cref{QuadratureLevels}, and where $\mathcal{V}_\ell$ and $R^2_\ell$ have been estimated using an independent sample of size \num{10000}.
The bold percentages reported in the last rows and last column correspond to the variance reduction factor of the MLMC-* estimators compared to the high-fidelity MC estimator.}
\label{tab:variance_reduction2}%
\begin{tabular}{clcccc}
    \toprule
    \multicolumn{2}{c}{$\ell$}  & 0 & 1 & 2 & 3 \\
    \midrule
    \multicolumn{2}{c}{$\V[Y_\ell]$}
    & \num{1.0850e+04} & \num{1.1153e+04} & \num{1.1205e+04} & \num{1.1218e+04} \\
    \multicolumn{2}{c}{$\mathcal{V}_\ell = \V[Y_\ell - Y_{\ell-1}]$}
    & \num{1.0850e+04} & \num{5.9029e+02} & \num{1.0590e+00} & \num{5.8160e-02} \\
    \midrule
    \multirow{5}{*}{$R^2_{\ell}$}
    & MLMC & 0 & 0 & 0 & 0 \\
    & MLMC-CV & 0.9838 & 0.9916 & 0.8224 & 0.6469 \\
    & MLMC-MLCV & 0.9840 & 0.9920 & 0.9173 & 0.9202 \\
    & MLMC-CV[0] & 0.9838 & 0 & 0 & 0 \\
    & MLMC-MLCV[0] & 0.9840 & 0.9916 & 0.8992 & 0.9027 \\
    \midrule\midrule
    \multirow{5}{*}{$n_\ell^* \dfrac{\mathcal{C}_\ell + \mathcal{C}_{\ell-1}}{\mathcal{C}}$}
    & MLMC & \qty{69.63}{\percent} & \qty{28.13}{\percent} & \qty{1.68}{\percent} & \qty{0.56}{\percent} \\
    & MLMC-CV & \qty{71.00}{\percent} & \qty{20.66}{\percent} & \qty{5.69}{\percent} & \qty{2.66}{\percent} \\
    & MLMC-MLCV & \qty{73.66}{\percent} & \qty{20.97}{\percent} & \qty{4.05}{\percent} & \qty{1.32}{\percent} \\
    & MLMC-CV[0] & \qty{22.59}{\percent} & \qty{71.69}{\percent} & \qty{4.29}{\percent} & \qty{1.42}{\percent} \\
    & MLMC-MLCV[0] & \qty{72.84}{\percent} & \qty{21.30}{\percent} & \qty{4.42}{\percent} & \qty{1.44}{\percent} \\
    \midrule
    \multirow{5}{*}{$\mathcal{S}_\ell^2$}
    & MLMC & 1356.31 & 2673.54 & 2766.49 & 2797.65 \\
    & MLMC-CV & 21.99 & 36.64 & 41.33 & 43.62 \\
    & MLMC-MLCV & 21.72 & 35.85 & 38.99 & 40.04 \\
    & MLMC-CV[0] & 21.99 & 382.87 & 418.54 & 430.71 \\
    & MLMC-MLCV[0] & 21.76 & 36.35 & 39.84 & 41.01 \\
    \midrule
    \multirow{5}{*}{$1-\dfrac{\mathcal{S}_\ell^2}{\mathcal{C}_\ell\V[Y_\ell]}$}
     & MLMC & \qty{0.00}{\percent} & \qty{4.11}{\percent} & \qty{50.62}{\percent} & \textbf{\qty{75.06}{\bpercent}} \\
    & MLMC-CV & \qty{98.38}{\percent} & \qty{98.69}{\percent} & \qty{99.26}{\percent} & \textbf{\qty{99.61}{\bpercent}} \\
    & MLMC-MLCV & \qty{98.40}{\percent} & \qty{98.71}{\percent} & \qty{99.30}{\percent} & \textbf{\qty{99.64}{\bpercent}} \\
    & MLMC-CV[0] & \qty{98.38}{\percent} & \qty{86.27}{\percent} & \qty{92.53}{\percent} & \textbf{\qty{96.16}{\bpercent}} \\
    & MLMC-MLCV[0] & \qty{98.40}{\percent} & \qty{98.70}{\percent} & \qty{99.29}{\percent} & \textbf{\qty{99.63}{\bpercent}} \\
    \bottomrule
\end{tabular}%
\end{table}

\paragraph{Accounting for the construction cost of the surrogate model}
These first results from \cref{fig:mc_vs_ml_without_offset} highlight the interest of multilevel control variates, be it with the MLMC-MLCV estimator or simply with the MLCV one. 
These results suppose that the PC models are not built specifically for the study, so that the budget does not include the number of $f_3$-equivalent simulations required for their construction.
\Cref{fig:mc_vs_ml_with_offset} illustrates the alternative case where the cost of the surrogate based CV estimators includes the cost of constructing the surrogates.
As a result, an additional budget of $\mathcal{C}^{\textnormal{DoE}}=400$ (see \cref{surrModel}), is allocated to the construction of the PC models.
As was the case for the single-level PC-based CV estimators, this results in an offset of 400 in the total cost of the MLCV and MLMC-MLCV estimators.
The effect is especially noticeable when the construction budget larger than the estimation budget, i.e.\ for $\mathcal{C} \in \{100; 300\}$.
For a sampling cost of 100, i.e.\ a total evaluation cost of 500, the MLCV estimator is still slightly more accurate than the MLMC estimator, while the MLMC-MLCV estimator still largely outperforms both.

\subsubsection{Variants of MLMC-based CV estimators}\label{res:mlmcmcv_vs_mlmf}

The discussion about the surrogate construction budget prompts us to investigate variants of MLMC-MLCV using fewer surrogates, in order to reduce the total evaluation cost for limited budgets.
The MLMC-CV estimator is described in \cref{sec:MLMF}, while the variants MLMC-CV[0] and MLMC-MLCV[0] of MLMC-CV and MLMC-MLCV are introduced in \cref{tabsummary_methods}.
The MLMC-CV[0] estimator only uses a control variate based on $g_0$, so that the construction cost drops to $\mathcal{C}^{\textnormal{DoE}}=100$, while the MLMC-MLCV[0] estimator uses control variates based on $g_0$, $g_1$ and $h_1$, so that the construction cost drops to $\mathcal{C}^{\textnormal{DoE}}=200$.
The MLMC-CV estimator uses control variates based on surrogates at all levels, so that the construction cost remains $\mathcal{C}^{\textnormal{DoE}}=400$.

\begin{figure}[!ht]
\centering
\begin{subfigure}[t]{0.49\textwidth}
\centering
\includegraphics[width=\textwidth]{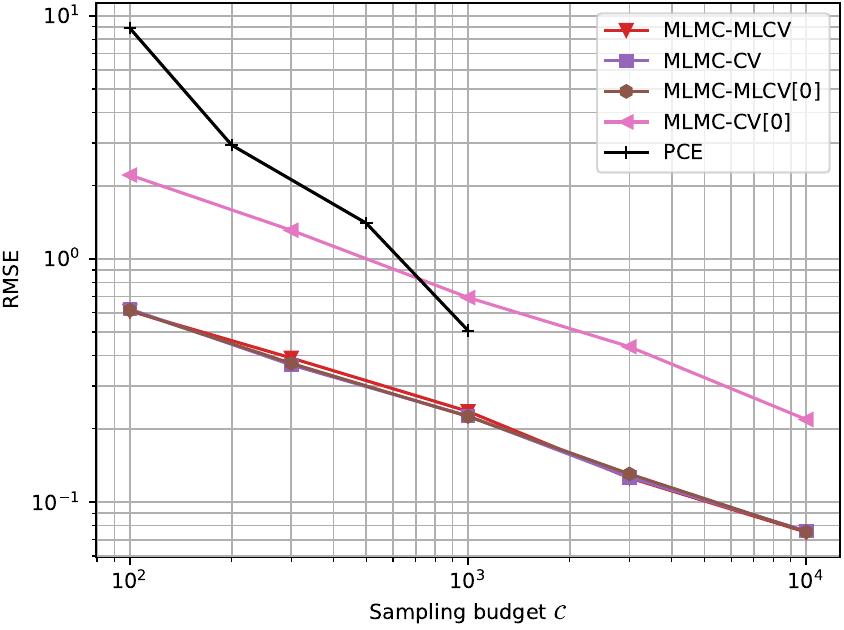}
\caption{The evaluation cost does not include the cost of construction of the surrogate models.}
\label{fig:cv_mlmc_1}
\end{subfigure}\hfill%
\begin{subfigure}[t]{0.49\textwidth}
\centering
\includegraphics[width=\textwidth]{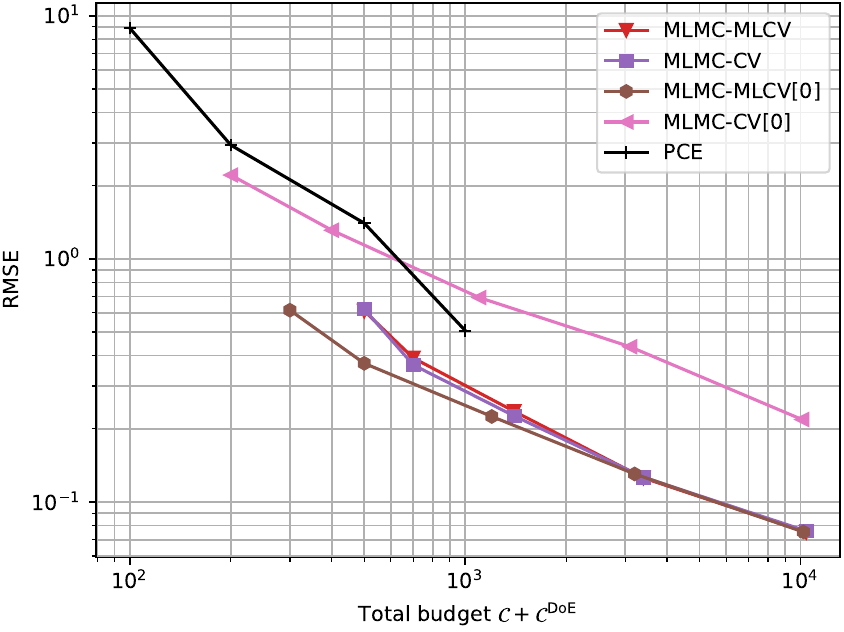}
\caption{The evaluation cost includes the cost of construction of the surrogate models.}
\label{fig:cv_mlmc_2} 
\end{subfigure}
\caption{RMSE of the MLMC-MLCV, MLMC-CV, MLMC-MLCV[0], and MLMC-CV[0] estimators of $\theta=\E[f_3(\vX)]$ (see \cref{eq:heateq_Y_discr}) with respect to the sampling budget $\mathcal{C}\in\{\num{100};\num{300};\num{1000};\num{3000};\num{10000}\}$ of $f_3$ evaluations. 
The surrogates used for the MLMC-based estimators are described in \cref{tab:pces_ml}.
The RMSEs of the pure PC estimators, summarized in \cref{tab:pure_pces} are shown in black for comparison.
The RMSEs for the MLMC-* estimators are computed using 500 replicates, while the RMSEs of the PC-only estimators are computed from 200 replicates.}
\label{fig:cv_mlmc}
\end{figure}

\Cref{fig:cv_mlmc_1} shows that MLMC-CV[0] has much higher RMSE than the other variants, resulting from the fact that it only reduces the variance associated with the coarsest level of the MLMC estimator.
This behavior is consistent with the quantities of \cref{tab:variance_reduction2}.
In particular, the value of $\mathcal{S}_L^2$ is about 10 times higher than for the other variants, accounting for its RMSE being about 3 times higher than for the other variants.
The remaining variants have similar performances regardless of the estimation budget, which is consistent with the $S_L^2$ values given in \cref{tab:variance_reduction2}. 
On the other hand, when considering the construction cost of the surrogates, \cref{fig:cv_mlmc_2} shows that MLMC-MLCV[0] performs best, as it uses only surrogate models related to the two coarsest levels, namely $g_0$, $g_1$ and $h_1$, so that the construction cost is reduced. 
Furthermore, these surrogates have excellent $Q^2$, and they are such that $g_0(\vX)$ is highly correlated with $Y_0$, $g_1(\vX)$ is highly correlated with $Y_1$, and $h_1(\vX)$ is highly correlated with $Y_1-Y_0$, $Y_2-Y_1$ and $Y_3-Y_2$.
Namely, the associated Pearson correlation coefficients reported in \cref{tab:mlmc_pearson} are all at least 0.94.
Therefore, should one have to build the surrogates specifically for the CV estimation of a statistic, it is more advantageous to adopt the MLMC-MLCV[0] variant over the others. 
In our case, the construction budget is divided by two compared to MLMC-MLCV and MLMC-CV, for a similar performance in terms of RMSE.
For completeness, we also compare the proposed estimators with plain PC models of the high-fidelity simulator $f_3$, whose characteristics are summarized in \cref{tab:pure_pces}. 
The black curve in \cref{fig:cv_mlmc} shows that, for the same budget, combining MLMC sampling and multifidelity surrogate models is more advantageous than using a surrogate model of the highest fidelity model alone. 
By extrapolating, the latter should eventually become better than the MLMC-* estimators around a budget of \num{3000}, which is much higher than practically affordable for expensive simulators. 
Note also that the PC models used here for comparison are optimized using the entire budget, while the MLMC-* estimators use PC models that are constructed from an arbitrary and non-optimized distribution of the sampling budget between the different model levels as reported in \cref{tab:pces_ml}. 
Strategies for optimizing this tradeoff, which constitutes a promising avenue for further improvement, will be investigated in future work.

\begin{table}[!htbp]
\caption{PC models of the high-fidelity simulator $f_L$ used directly for the estimation of $\E[\mathcal{M}(\vX)]$, with their sample size, degree and quality measure. These models are built from an optimized LHS DoE, using the basis-adaptive LARS algorithm of~\cite[Fig.~5]{Blatman2011_AdaptiveSparsePolynomial}.
The average values of $p^*$, $|\tilde{\mathcal{A}}_{p^*}|$ and $Q^2$ over 200 replications is reported, along with their respective standard deviation between brackets.}
\label{tab:pure_pces}
\centering
\begin{tabular}[c]{ccccc}
\toprule
$n_L^{\textnormal{DoE}}$ & 100 & 200 & 500 & 1000 \\
\midrule
$p^*$ & \num{4(2)} & \num{6(1)} & \num{9(1)} & \num{11(1)}  \\
$|\tilde{\mathcal{A}}_{p^*}|$ & \num{20(13)} & \num{48(16)} & \num{118(31)} & \num{219(56)}  \\
$Q^2$ & \num{0.35(0.21)} & \num{0.84(0.02)} & \num{0.95(0.01)} & \num{0.98(0.00)}  \\
\bottomrule
\end{tabular}
\end{table}

\subsubsection{Budget allocation}\label{res:budget_alloc}

Lastly, \cref{fig:allocation_calls} shows the number $n_\ell$ of evaluations for each of the $L+1$ correction levels of the MLMC-* telescopic sum. 
Precisely, $n_0$ is the number of evaluations of $f_0$, while $n_\ell$ is the number of evaluations of $f_{\ell}-f_{\ell-1}$, for $\ell >0$. 
We observe a typical sample allocation for MLMC-like estimators in ideal cases, that is, many coarse evaluations, and fewer and fewer fine evaluations.
The MLMC-CV[0] estimator slightly deviates from this pattern, with $n_1 \approx n_0$, which can be explained by the fact that $\mathcal{V}_1 \approx 590$ is of the same order of magnitude as $(1-R_0^2) \mathcal{V}_0 \approx 176$.
\Cref{fig:allocation_cost} depicts the share of overall sampling cost associated with the different correction levels.
Specifically, $n_0\mathcal{C}_0\mathcal{C}^{-1}$ is the share of sampling budget dedicated to evaluating $f_0$, and $n_\ell(\mathcal{C}_\ell + \mathcal{C}_{\ell-1})\mathcal{C}^{-1}$ is the share dedicated to evaluating $f_\ell - f_{\ell-1}$, for $\ell>0$. 
We see that, except for MLMC-CV[0] for the reasons explained above, most of the sampling budget (around 70\%) is allocated to the coarsest level, and most of the remaining budget is dedicated to correction level $\ell=1$.
We note that the CV-based MLMC estimators dedicate slightly more budget to levels 2 and 3 than for the standard MLMC estimator.
The sampling budget shares of \cref{fig:allocation_cost} are consistent with the theoretical optimal shared reported in \cref{tab:variance_reduction2}, including those of the MLMC-CV[0] estimator.
This suggests that the sample allocation resulting from \cref{algoMLCV} seems to converge to the theoretical optimal sample allocation \cref{eq:nl_opt}.

\begin{figure}[!ht]
\begin{subfigure}[t]{0.49\textwidth}
\centering
\includegraphics[width=\textwidth]{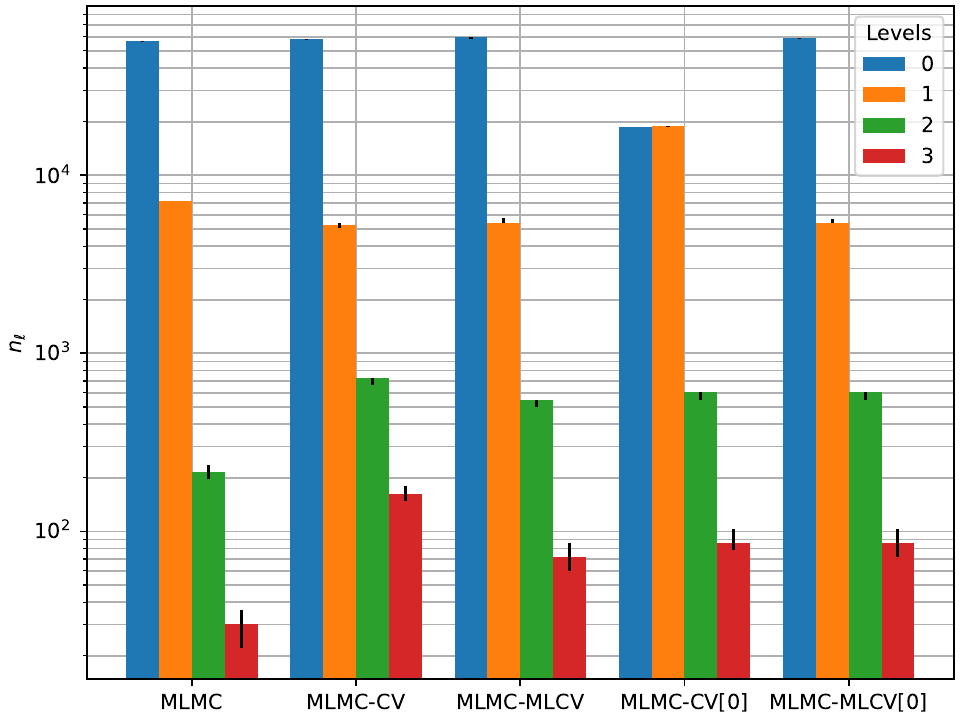}
\caption{Sample allocation. The bars represent the median sample size $n_{\ell}$ associated with the corresponding correction level $\ell$ of the MLMC-* telescopic sum. The black vertical lines represent the 25\% and 75\% quantiles.}
\label{fig:allocation_calls}
\end{subfigure}\hfill%
\begin{subfigure}[t]{0.49\textwidth}
\centering
\includegraphics[width=\textwidth]{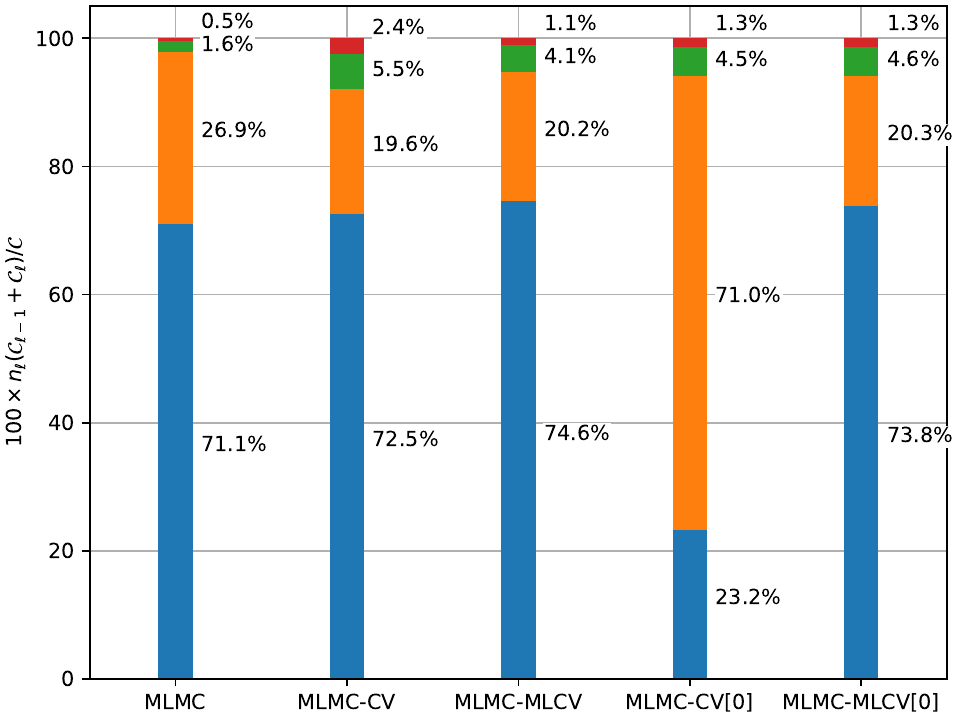}
\caption{Budget allocation. 
The stacked bars represent the proportion of the sampling cost dedicated to each level of the MLMC-* telescopic sum.}
\label{fig:allocation_cost} 
\end{subfigure}
\caption{Sample allocation and associated computational cost across the correction levels for the MLMC-based estimators of $\theta=\E[f_3(\vX)]$ (see \cref{eq:heateq_Y_discr}), with an estimation budget of $\mathcal{C}=\num{10000}$.}
\end{figure}

\section{Conclusions}\label{conclusion}
In this paper, we proposed multilevel variance reduction strategies relying on surrogate-based control variates.
On the one hand, using specific surrogate models, such as polynomial chaos expansions or Taylor polynomial expansions, allows to directly access exact statistics (mean, variance) of the control variates.
Even if these exact statistics are not directly accessible (e.g., when using GPs), they can be estimated very accurately at negligible cost.
This contrasts with typical control variates relying on lower-fidelity models based on models/simulators with degraded physics or coarser discretizations, for which approximate control variate strategies need to be devised, resulting in a lower variance reduction.
On the other hand, when multiple levels of fidelities (e.g., based on the discretization) with a clear cost/accuracy hierarchy are available, the surrogate-based control variate approach can be efficiently combined with multilevel strategies.
The first strategy, MLCV, simply consists of using multiple control variates based on surrogate models of the simulators corresponding to the different levels.
The main advantage is that the surrogate models corresponding to coarse levels may be constructed using larger sample sizes than for the finest level, resulting in a more accurate surrogate model (i.e., with lower model error).
This strategy thus leads to a greater variance reduction compared to only using one surrogate model based on the finest level.
This is supported by the numerical experiments we conducted, as well as by the theoretical variance reduction provided by \cref{prop:add_cv}.
The second strategy, MLMC-MLCV, allows to further improve the variance reduction by combining the surrogate-based control variates with an MLMC strategy.
The most appropriate way to construct and utilize the surrogate models is, however, not straightforward, and was discussed in detail in \cref{sec:MLMF,sec:MLMC-MCV}.
The additional variance reduction as compared to plain MLMC is demonstrated in our numerical experiments and supported by the theoretical variance reduction factor \cref{eq:R2_MLMC-MLCV} derived in \cref{sec:MLMC-MCV}.

The construction cost of the surrogate models was discussed from two perspectives.
When the surrogate models are constructed for the sole purpose of serving for the control variate estimation, then the cost of their construction must be taken into account for fair comparison with other approaches.
In such a case, it may not be optimal, in terms of cost/accuracy tradeoff, to construct surrogate models on all levels, especially when only a limited budget is available.
In particular, using a subset of surrogate models based on the coarser levels may already lead to considerable variance reduction, provided that the outputs of the coarse surrogate models are sufficiently correlated with that of the high-fidelity simulator.
Then, considering additional surrogate models on finer levels might only result in marginal improvement, at the expense of a significant computational cost.
On the contrary, if the surrogate models have already been constructed for other purposes, and are, in some sense, available ``for free,'' then their construction cost need not be considered, and the entire set of surrogate models may then be used.

From the former perspective, that is when the cost of the surrogate construction is considered as part of the estimation cost, one may devise more involved strategies seeking to optimize the tradeoff between the construction cost and the model error, directly impacting the projected variance reduction.
For instance, in the context of polynomial chaos surrogate models, the truncation strategy may be controlled to this end, as proposed in a stochastic Galerkin framework in~\cite{Yang2022_ControlVariatePolynomial}, where the total polynomial degree is optimized alongside the sample size and the CV parameter to minimize the PC-based CV estimator's variance under a cost constraint.
Another avenue to improve the proposed approach would be to replace the MLMC part of the MLMC-MLCV strategy by a more efficient multilevel approach, such as the multilevel best linear unbiased estimator (MLBLUE)~\cite{Schaden2020_MultilevelBestLinear, Schaden2021_AsymptoticAnalysisMultilevel, Schaden2021_thesis}.
In particular, an MLBLUE-MLCV strategy should be more efficient when a collection of low-fidelity simulators (e.g., with degraded physics) with no clear cost/accuracy hierarchy is available.
Finally, although the proposed approaches apply to the estimation of arbitrary statistics, they were only tested here on the estimation of expected values.
The theoretical and algorithmic ingredients for the estimation of variances are, however, described in this paper and may be tested in follow-up investigations and numerical experiments.
Specifically, the multifidelity estimation of variance-based sensitivity indices is of particular interest to our team.

\section{Acknowledgment}
We wish to acknowledge the PIA framework
(CGI, ANR) and the industrial members of the IRT
Saint Exupéry project R-Evol: Airbus, Liebherr,
Altran Technologies, Capgemini DEMS France,
CENAERO and Cerfacs for their support,
financial funding and own knowledge.



\clearpage
\newpage
\appendix
\section{Optimal parameter for expectation and variance CV estimators}
\label{app:CVmeanvar}
We derive here the expression of the optimal CV parameter $\boldsymbol{\alpha}^*$ for the CV estimators of the expectation and of the variance.
We define $Y=f(\vX)$, $Z_m=g_m(\vX)$ and $\vZ=(Z_m)_{m=1}^M$.
Then, given an input $n$-sample $\{\vX^{(i)}\}_{i=1}^n$, we define $Y^{(i)}=f(\vX^{(i)})$, $Z_m^{(i)}=g_m(\vX^{(i)})$ and $\vZ^{(i)}=(Z_m^{(i)})_{m=1}^M$.

\subsection{CV estimator of the expectation}
\label{app:CVmeanvar-mean}
For the expectation, we have $\boldsymbol{\Sigma}=\C[\hat{E}[\vZ]]$ and $\mathbf{c} = \C[\hat{E}[Y], \hat{E}[\vZ]]$, i.e.
\begin{align}
    [\boldsymbol{\Sigma}]_{m,m'} 
    & =
    \C[\hat{E}[Z_m], \hat{E}[Z_{m'}]]
    =
    n^{-2} \sum_{i,j =1}^n\C[Z_m^{(i)},Z_{m'}^{(j)}]
    = 
    n^{-1} \C[Z_m,Z_{m'}],\\[.5em]
    [\mathbf{c}]_m 
    & =
    \C[\hat{E}[Y], \hat{E}[Z_m]]
    =
    n^{-2} \sum_{i,j =1}^n\C[Y^{(i)},Z_m^{(j)}]
    =
    n^{-1} \C[Y,Z_m],
\end{align}
so that $\boldsymbol{\Sigma} = n^{-1}\C[\vZ]$, $\mathbf{c} = n^{-1}\C[Y,\vZ]$, and, eventually,
\begin{equation}
    \boldsymbol{\alpha}^*
    =
    \C[\mathbf{Z}]^{-1} \C[Y,\mathbf{Z}].
\end{equation}
Furthermore, we have 
$\V[\hat{E}[Y]] = n^{-1}\V[Y]$, so that
\begin{equation}
    R^2 
    =
    \dfrac{\C[Y, \vZ]^\intercal \C[\vZ]^{-1} \C[Y, \vZ]}{\V[Y]}
    =
    \vec{r}_{Y,\vZ}^\intercal \mat{R}_{\vZ}^{-1} \vec{r}_{Y,\vZ},
\end{equation}
where
\begin{equation}
    \vec{r}_{Y,\vZ} = (\V[Y]\mat{D}_{\vZ})^{-1/2} \C[Y,\vZ], \quad
    \mat{R}_{\vZ} = \mat{D}_{\vZ}^{-1/2} \C[\vZ] \mat{D}_{\vZ}^{-1/2}, \quad
    \mat{D}_{\vZ} = \Diag(\C[\vZ]).
\end{equation}
Thus, $R^2 \in [0,1]$ corresponds to the squared coefficient of multiple correlation between $Y$ and the control variates $Z_1, \ldots, Z_M$.

\subsection{CV estimator of the variance}
\label{app:CVmeanvar-var}
Similarly, for the variance, we have $\boldsymbol{\Sigma}=\C[\hat{V}[\vZ]]$ and $\mathbf{c} = \C[\hat{V}[Y], \hat{V}[\vZ]]$. We start by deriving useful identities. In what follows, for any random variable $A$, we denote the corresponding centered variable by $\bar{A} := A - \E[A]$.
First, we remark that $\hat{V}[A] = \hat{V}[\bar{A}]$, so that, for any two random variables $Y$ and $Z$,
\begin{equation}
    \C[\hat{V}[Y], \hat{V}[Z]]
    =
    \E[\hat{V}[\bar{Y}] \hat{V}[\bar{Z}]]
    -
    \V[Y]\V[Z].
\end{equation}
Furthermore, it can be shown that
\begin{equation}
    \E[\hat{V}[\bar{Y}] \hat{V}[\bar{Z}]]
    =
    \left(\dfrac{n}{n-1}\right)^2
    \left(
    a_n(\bar{Y}, \bar{Z})
    +
    b_n(\bar{Y}, \bar{Z})
    - 
    c_n(\bar{Y}, \bar{Z})
    -
    c_n(\bar{Z}, \bar{Y})
    \right),
\end{equation}
with (see proof below)
\begin{align}
    a_n(\bar{Y}, \bar{Z})
    & :=
     \E[ \hat{E}[\bar{Y}^2] \hat{E}[\bar{Z}^2] ]
     =
     \dfrac{1}{n} \C[\bar{Y}^2,\bar{Z}^2]
     + \V[Y]\V[Z]
  =  a_n(\bar{Z}, \bar{Y}), \label{eq:cv_var_an}\\
  b_n(\bar{Y}, \bar{Z})
  & :=
  \E[ \hat{E}[\bar{Y}]^2 \hat{E}[\bar{Z}]^2 ]
  =
  \dfrac{a_n(\bar{Y}, \bar{Z})}{n^2} 
  + 2 \dfrac{n - 1}{n^3} \C[Y, Z]^2
  = b_n(\bar{Z}, \bar{Y}), \label{eq:cv_var_bn}\\
  c_n(\bar{Y}, \bar{Z})
  & :=
  \E[ \hat{E}[\bar{Y}^2] \hat{E}[\bar{Z}]^2 ]
  =
  \dfrac{a_n(\bar{Y}, \bar{Z})}{n}
  = c_n(\bar{Z}, \bar{Y}). \label{eq:cv_var_cn}
\end{align}
We thus have
\begin{align}
    \C[\hat{V}[Y], \hat{V}[Z]]
    & =
    \dfrac{1}{n} \C[\bar{Y}^2, \bar{Z}^2] + \dfrac{2}{n(n-1)} \C[Y, Z]^2,\\
    \V[\hat{V}[Y]]
    =
    \C[\hat{V}[Y], \hat{V}[Y]]
    & =
    \dfrac{1}{n} \V[\bar{Y}^2] + \dfrac{2}{n(n-1)} \V[Y]^2,
\end{align}
eventually leading to
\begin{align}
    \boldsymbol{\Sigma}
    & =
    \C[\hat{V}[\vZ]]
    =
    \dfrac{1}{n} \left(
    \C[\bar{\vZ}^{\odot 2}] + \dfrac{2}{n-1} \C[\vZ]^{\odot 2}
    \right),\\[.5em]
    \mathbf{c}
    & =
    \C[\hat{V}[Y], \hat{V}[\vZ]]
    =
    \dfrac{1}{n} \left(
    \C[\bar{Y}^2, \bar{\vZ}^{\odot 2}] + \dfrac{2}{n-1} \C[Y, \vZ]^{\odot 2}
    \right),
\end{align}
so that
\begin{align}
    \boldsymbol{\alpha}^*
    & =
    \left[
    \C[\bar{\vZ}^{\odot 2}] + \dfrac{2}{n-1} \C[\vZ]^{\odot 2}
    \right]^{-1}
    \left[
    \C[\bar{Y}^2, \bar{\vZ}^{\odot 2}] + \dfrac{2}{n-1} \C[Y, \vZ]^{\odot 2}
    \right], \\[1em]
    R^2
    & =
    \dfrac{
    \left[
    \C[\bar{Y}^2, \bar{\vZ}^{\odot 2}] + \dfrac{2}{n-1} \C[Y, \vZ]^{\odot 2}
    \right]^\intercal
    \boldsymbol{\alpha}^*
    }
    {\V[\bar{Y}^2] + \dfrac{2}{n-1} \V[Y]^2}.
\end{align}
As $n \to \infty$, we see that 
\begin{align}
    \boldsymbol{\alpha}^* & \to \C[\bar{\vZ}^{\odot 2}]^{-1}\C[\bar{Y}^2, \bar{\vZ}^{\odot 2}] \\[.5em]
    R^2 &\to \V[\bar{Y}^2]^{-1} \C[\bar{Y}^2, \bar{\vZ}^{\odot 2}]^\intercal\C[\bar{\vZ}^{\odot 2}]^{-1}\C[\bar{Y}^2, \bar{\vZ}^{\odot 2}]
    =
    \vec{r}_{Y,\vZ}^\intercal \mat{R}_{\vZ}^{-1} \vec{r}_{Y,\vZ} =: R^2_{\textnormal{lim}}, 
    \label{eq:correl_coeff_variance}
\end{align}
where
\begin{equation}
    \vec{r}_{Y,\vZ} = (\V[\bar{Y}^2]\mat{D}_{\vZ})^{-1/2} \C[\bar{Y}^2, \bar{\vZ}^{\odot 2}], \quad
    \mat{R}_{\vZ} = \mat{D}_{\vZ}^{-1/2} \C[\bar{\vZ}^{\odot 2}] \mat{D}_{\vZ}^{-1/2},
\end{equation}
with $\mat{D}_{\vZ} = \Diag(\C[\bar{\vZ}^{\odot 2}])$.
Thus, $R^2_{\textnormal{lim}} \in [0,1]$ corresponds to the squared coefficient of multiple correlation between $\bar{Y}^2$ and $\bar{Z}_1^2, \ldots, \bar{Z}_M^2$.

We now proceed to the proof of identities~\cref{eq:cv_var_an,eq:cv_var_bn,eq:cv_var_cn}.
First, for \cref{eq:cv_var_an}, by definition
\begin{equation}
  a_n(\bar{Y}, \bar{Z})
  = \dfrac{1}{n^2} \E[ \textstyle\sum_{i=1}^{n} (\bar{Y}^{(i)})^2 \sum_{i=1}^{n} (\bar{Z}^{(i)})^2 ]
  = \dfrac{1}{n^2} \sum_{i,j=1}^{n} \E[(\bar{Y}^{(i)})^2(\bar{Z}^{(j)})^2].
\end{equation}
We distinguish two (disjoint) cases:
\begin{enumerate}
\item $i = j$:  $\E[(\bar{Y}^{(i)})^2(\bar{Z}^{(j)})^2] = \E[(\bar{Y}^{(i)})^2(\bar{Z}^{(i)})^2] 
  =  \E[\bar{Y}^2\bar{Z}^2] 
  = \C[\bar{Y}^2, \bar{Z}^2] + \V[Y]\V[Z]$.
  There are $n$ such terms in the sum. 
\item $i \ne j$: $\E[(\bar{Y}^{(i)})^2(\bar{Z}^{(j)})^2] 
  = \E[(\bar{Y}^{(i)})^2]\E[(\bar{Z}^{(j)})^2] 
  = \E[\bar{Y}^2]\E[\bar{Z}^2] = \V[Y]\V[Z]$.
  There are $n(n-1)$ such terms in the sum. 
\end{enumerate}
Then \cref{eq:cv_var_an} follows.
For \cref{eq:cv_var_bn},
\begin{equation}
    b_n(\bar{Y}, \bar{Z})
    =
    \E[ ( \textstyle\sum_{i=1}^{n} \bar{Y}^{(i)} )^2 ( \sum_{i=1}^{n} \bar{Z}^{(i)} )^2]
  = \dfrac{1}{n^4} \sum_{i,j,k,\ell=1}^{n} \E[\bar{Y}^{(i)}\bar{Y}^{(j)}\bar{Z}^{(k)}\bar{Z}^{(\ell)}].
\end{equation}
We distinguish five (disjoint) cases:
\begin{enumerate}
\item $i=j=k=\ell$: $\E[\bar{Y}^{(i)}\bar{Y}^{(j)}\bar{Z}^{(k)}\bar{Z}^{(\ell)}] = \E[(\bar{Y}^{(i)})^2(\bar{Z}^{(i)})^2] = \E[\bar{Y}^2\bar{Z}^2] = \C[\bar{Y}^2, \bar{Z}^2] + \V[Y]\V[Z]$. There are $n$ such terms in the sum.
\item $i=j \ne k=\ell$: $\E[\bar{Y}^{(i)}\bar{Y}^{(j)}\bar{Z}^{(k)}\bar{Z}^{(\ell)}] = \E[(\bar{Y}^{(i)})^2]\E[(\bar{Z}^{(k)})^2] =  \E[\bar{Y}^2]\E[\bar{Z}^2] = \V[Y]\V[Z]$. There are $n(n-1)$ such terms in the sum.
\item $i=k \ne j=\ell$: $\E[\bar{Y}^{(i)}\bar{Y}^{(j)}\bar{Z}^{(k)}\bar{Z}^{(\ell)}] = \E[\bar{Y}^{(i)}\bar{Z}^{(i)}] \E[\bar{Y}^{(j)}\bar{Z}^{(j)}] =  \E[\bar{Y} \bar{Z}]^2 = \C[Y, Z]^2$. There are $n(n-1)$ such terms in the sum.
\item $i=\ell \ne j=k$: $\E[\bar{Y}^{(i)}\bar{Y}^{(j)}\bar{Z}^{(k)}\bar{Z}^{(\ell)}] = \E[\bar{Y}^{(i)}\bar{Z}^{(i)}] \E[\bar{Y}^{(j)}\bar{Z}^{(j)}] =  \E[\bar{Y} \bar{Z}]^2 = \C[Y, Z]^2$. There are $n(n-1)$ such terms in the sum.
\item All remaining cases (at least one of the indices $i,j,k,\ell$ is different from all the others): $\E[\bar{Y}^{(i)}\bar{Y}^{(j)}\bar{Z}^{(k)}\bar{Z}^{(\ell)}] = 0$.
\end{enumerate}
Then \cref{eq:cv_var_bn} follows.
Finally, for \cref{eq:cv_var_cn},
\begin{equation}
    c_n(\bar{Y}, \bar{Z})
    \E[  ( \textstyle\sum_{i=1}^{n} (\bar{Y}^{(i)})^2 ) ( \sum_{i=1}^{n} \bar{Z}^{(i)} )^2 ]
    = \dfrac{1}{n^3} \sum_{i,j,k=1}^{n} \E[(\bar{Y}^{(i)})^2\bar{Z}^{(j)}\bar{Z}^{(k)}].
\end{equation}
We distinguish three (disjoint) cases:
\begin{enumerate}
\item $i=j=k$: $\E[(\bar{Y}^{(i)})^2\bar{Z}^{(j)}\bar{Z}^{(k)}] = \E[(\bar{Y}^{(i)})^2(\bar{Z}^{(i)})^2] = \E[\bar{Y}^2\bar{Z}^2] = \M^4[Y, Z]$. There are $n$ such terms in the sum.
\item $i \ne j = k$: $\E[(\bar{Y}^{(i)})^2\bar{Z}^{(j)}\bar{Z}^{(k)}] = \E[(\bar{Y}^{(i)})^2]\E[(\bar{Z}^{(j)})^2] = \E[\bar{Y}^2]\E[\bar{Z}^2] = \V[Y]\V[Z]$. There are $n(n-1)$ such terms in the sum.
\item All remaining cases (at least one of the indices $j,k$ is different from all the others): \\${\E[(\bar{Y}^{(i)})^2\bar{Z}^{(j)}\bar{Z}^{(k)}] = 0}$.
\end{enumerate}
Then \cref{eq:cv_var_cn} follows.

\begin{remark}
When the expected value $\boldsymbol{\mu}_{\vZ}$ of $\vZ$ is known, as is the case when defining $\vZ$ from the prediction of certain surrogate models, such as PC expansion and Taylor polynomials (and, in some instances, GPs, see \cref{sec:GP}), it is possible to replace $\hat{V}[\vZ]$ with $\hat{E}[\bar{\vZ}^{\odot 2}]$. The derivation of the optimal CV parameter is somewhat easier and leads to similar results as in the unknown expectation case.
Specifically,
\begin{align}
  [\boldsymbol{\Sigma}]_{m,m'} 
    &= \C[\hat{E}[\bar{Z}_m^2], \hat{E}[\bar{Z}_{m'}^2]]   
    = \C[n^{-1} \textstyle \sum_{i=1}^n (\bar{Z}_m^{(i)})^2 ,n^{-1} \sum_{i=1}^n (\bar{Z}_{m'}^{(i)})^2]\\[.5em]
    & = n^{-2} \textstyle\sum_{i,j=1}^{n} \C[(\bar{Z}_m^{(i)})^2, (\bar{Z}_{m'}^{(j)})^2]
    = n^{-2} \textstyle\sum_{i=1}^{n} \C[(\bar{Z}_m^{(i)})^2, (\bar{Z}_{m'}^{(i)})^2]\\[.5em]
    & = n^{-1} \C[\bar{Z}_m^2, \bar{Z}_{m'}^2],
\end{align}
i.e.\ $\boldsymbol{\Sigma} = n^{-1} \C[\bar{\vZ}^{\odot 2}]$.
Regarding the vector of covariances $\mathbf{c}$, 
\begin{align}
  [\vec{c}]_m 
  &= \C[ \hat{V}[\bar{Y}],\hat{E}[\bar{Z}_m^2]]
  = \E[ \hat{V}[\bar{Y}] \hat{E}[\bar{Z}_m^2]] - \V[Y]\V[Z_m]\\[.5em]
  &= \dfrac{n}{n-1} (\E[ \hat{E}[\bar{Y}^2] \hat{E}[\bar{Z}_m^2]] - \E[ \hat{E}[\bar{Y}]^2 \hat{E}[\bar{Z}_m^2]] ) - \V[Y]\V[Z_m]\\[.5em]
  & = \dfrac{n}{n-1} ( a_n(\bar{Y}, \bar{Z}_m) - c_n(\bar{Z}_m, \bar{Y}) ) - \V[Y]\V[Z_m] 
  = a_n(\bar{Y}, \bar{Z}_m) - \V[Y]\V[Z_m] \\[.5em]
  &= n^{-1} \C[\bar{Y}^2, \bar{Z}_m^2],
\end{align}
i.e.\ $\vec{c} = n^{-1} \C[\bar{Y}^2, \bar{\vZ}^{\odot 2}]$, so that 
\begin{align}
    \boldsymbol{\alpha}^* &= \C[\bar{\vZ}^{\odot 2}]^{-1} \C[\bar{Y}^2, \bar{\vZ}^{\odot 2}],\\[1em]
    R^2 & = 
    \dfrac{
        \C[\bar{Y}^2, \bar{\vZ}^{\odot 2}]^\intercal \C[\bar{\vZ}^{\odot 2}]^{-1} \C[\bar{Y}^2, \bar{\vZ}^{\odot 2}]
    }
    {\V[\bar{Y}^2] + \dfrac{2}{n-1} \V[Y]^2}
    \xrightarrow[n \to \infty]{} R^2_{\textnormal{lim}},
\end{align}
with the same definitions as in \cref{eq:correl_coeff_variance}.
\end{remark}

\section{Variance reduction of bi-fidelity {(A)CV} estimators of the expectation}%
\label{app:cv_vs_acv}%

In this appendix, we illustrate the superiority of the (exact) CV estimator of the expectation over the bi-fidelity ACV estimator introduced in~\cite[section~3]{Ng2014_MultifidelityApproachesOptimization} (see also \cite[section~{IV}]{Geraci2017_MultifidelityMultilevelMonte}), hereafter referred to as the MFMC estimator, in terms of variance reduction.
We consider here a high-fidelity simulator $f$ and a low-fidelity version $g$.
We denote by $\rho$ the Pearson correlation coefficient between $f(\vX)$ and $g(\vX)$, and we define $w := c_g / c_f$ as the ratio of the expected cost $c_g$ of evaluating $g(\vX)$ to the expected cost $c_f$ of evaluating $f(\vX)$.
Note that our definition of $w$ is the inverse of the definition of $w$ in \cite{Ng2014_MultifidelityApproachesOptimization, Geraci2017_MultifidelityMultilevelMonte}, allowing our $w$ to remain finite when the expected cost of evaluating $g(\vX)$ vanishes.
We shall assume hereafter that $w, \rho^2 \in (0,1)$.

\subsection{The MFMC estimator}
For a given (expected) computational budget $\mathcal{C}$, the optimal MFMC estimator~\cite{Ng2014_MultifidelityApproachesOptimization, Geraci2017_MultifidelityMultilevelMonte} reads
\begin{equation}\label{eq:app_MFMC}
    \hat{\mu}_{\mathcal{C}}^{\textnormal{MFMC}} 
    =
    \frac{1}{N} \sum_{i=1}^N f(\vX^{(i)}) 
    - 
    \alpha
    \left(
        \frac{1}{N} \sum_{i=1}^N g(\vX^{(i)}) 
        -
        \frac{1}{N'} \sum_{i=1}^{N'} g(\vX^{(i)}) 
    \right),
\end{equation}
where
\begin{equation}\label{eq:app_MFMC_optim_param}
    \alpha = \frac{\C[Y,Z]}{\V[Z]},
    \quad
    N = \frac{\mathcal{C}}{c_f (1+\eta w)},
    \quad
    \eta = \sqrt{\frac{\rho^2}{w(1-\rho^2)}},
    \quad
    N' = \eta N,
\end{equation}
and where $\mathcal{X} := \{\vX^{(1)}, \ldots, \vX^{(N')}\}$ is an ${N'}$-sample of $\vX$ (we assume here that ${N'}>N$; see \cite{Ng2014_MultifidelityApproachesOptimization} and \cref{coro:mfmc_eta2} for a discussion on this assumption).
Note that $N$ and $N'$ need to be appropriately rounded to integers for the estimator \cref{eq:app_MFMC} to make sense.
For the same (expected) budget $\mathcal{C}$, one can afford $p := \mathcal{C}/c_f$ evaluations of $f(\vX)$, i.e., the expected cost of $p$ evaluations of $f(\vX)$ is $\mathcal{C}$.
Let $\hat{\mu}_p := p^{-1}\sum_{i=1}^p f(\vX^{(i)})$ denote the standard MC estimator of the expected value of $f(\vX)$ using a $p$-sample of $\vX$.
We now proceed to proving some useful properties of the MFMC estimator \cref{eq:app_MFMC}.

\begin{proposition}\label{prop:mfmc_beta}
    The variance of the MFMC estimator defined by \cref{eq:app_MFMC,eq:app_MFMC_optim_param} is such that
    $\V[\hat{\mu}_{\mathcal{C}}^{\textnormal{MFMC}}] = \beta^{\textnormal{MFMC}} \V[\hat{\mu}_p]$,
    with
    $\beta^{\textnormal{MFMC}} = (1+\eta w)^2 (1-\rho^2)$.
\end{proposition}
\begin{proof}
    From \cite[Eq.~{(6)}]{Ng2014_MultifidelityApproachesOptimization} (see also \cite[Eq.~{(9)}]{Geraci2017_MultifidelityMultilevelMonte}), we have 
    \begin{equation}
        \beta^{\textnormal{MFMC}} 
        =
        (1+\eta w) \left[ 1 - \left( 1 - \frac{1}{\eta} \right) \rho^2 \right].
    \end{equation}
    Furthermore,
    \begin{equation}
        1 - \left( 1 - \frac{1}{\eta} \right) \rho^2
        = 
        1 - \rho^2  + \frac{\rho^2}{\eta}
        =
        (1 - \rho^2) \left( 1 + \frac{1}{\eta} \frac{\rho^2}{1-\rho^2} \right).
    \end{equation}
    Finally, injecting the expression of $\eta$ defined in \cref{eq:app_MFMC_optim_param}, we notice that
    \begin{equation}
        \frac{1}{\eta} \frac{\rho^2}{1-\rho^2}
        =
        \sqrt{\frac{w(1-\rho^2)}{\rho^2}} \frac{\rho^2}{1-\rho^2}
        =
        \sqrt{\frac{w\rho^2}{1-\rho^2}}
        =
        \eta w,
    \end{equation}
    which concludes the proof.
\end{proof}
We note that $\beta^{\textnormal{MFMC}}$ does not depend on $\mathcal{C}$, and that the smaller $\beta^{\textnormal{MFMC}}$, the greater the variance reduction with respect to the standard, high-fidelity MC estimator for the same (expected) budget.
We also note that variance reduction is achieved if and only if $\beta^{\textnormal{MFMC}} < 1$.

\begin{proposition}\label{prop:mfmc_eta}
$\eta > 1$ (i.e., $N'>N$) if and only if $\rho^2 > w / (1+w)$.
\end{proposition}
\begin{proof}
    $\eta >1 \iff \eta^2 > 1 \iff \rho^2 > w(1-\rho^2) \iff (1+w)\rho^2 > w$, hence the result follows. 
\end{proof}

\begin{proposition}\label{prop:mfmc_var_red}
    $\beta^{\textnormal{MFMC}} < 1$ if and only if $\rho^2 > 4w/(1+w)^2$.
\end{proposition}
\begin{proof}
    We start by noticing that $1/(1-\rho^2) = 1+ \rho^2/(1-\rho^2) = 1+ \eta^2 w$.
    Then, it follows that 
    $\beta^{\textnormal{MFMC}} < 1 \iff (1+\eta w)^2 < 1/(1-\rho^2) = 1 + \eta^2 w \iff [2 - (1-w)\eta]\eta w < 0$.
    Because $w, \eta > 0$ and $(1-w) \in (0,1)$, we deduce that 
    $\beta^{\textnormal{MFMC}} < 1 \iff \eta > 2 / (1-w) \iff \eta^2 > 4 / (1-w)^2$.
    Injecting the expression of $\eta$ defined in \cref{eq:app_MFMC_optim_param}, we get 
    $\beta^{\textnormal{MFMC}} < 1 \iff \rho^2 (1-w)^2 > 4w (1-\rho^2)$, eventually leading to the desired result.
\end{proof}

\begin{proposition}\label{prop:mfmc_feasible}
    For any $w \in (0,1)$, $4w/(1+w)^2 < 1$.
\end{proposition}
\begin{proof}
    The function $a \colon x \mapsto 4x/(1+x)^2$ is continuous on $[0,1]$, and its derivative is defined on $[0,1]$ by $a'(x) = 4(1-x)/(1+x)^3$, which is positive for $x \in (0,1)$, so $a$ is strictly increasing on $(0,1)$.
    Furthermore, for $x \in [0,1]$, $a(x) = 1 \iff x = 1$, implying that, for any $x \in (0,1)$, $a(x) < 1$, which concludes the proof. 
\end{proof}
\cref{prop:mfmc_feasible} implies that, for any $w \in (0,1)$, variance reduction can always be achieved provided that $f(\vX)$ and $g(\vX)$ are sufficiently correlated.

\begin{proposition}\label{prop:comp_rho2_thresh}
    For any $w \in (0,1)$, $4w/(1+w)^2 >  2w / (1+w)$.
\end{proposition}
\begin{proof}
    Taking the ratio of the left-hand side to the right-hand side, we have
    \begin{equation}
        \dfrac{4w/(1+w)^2}{2w / (1+w)} = \dfrac{2}{1+w} > 1,
    \end{equation}
    since $w < 1$.
\end{proof}

\begin{corollary}\label{coro:mfmc_eta2}
    If $\beta^{\textnormal{MFMC}} < 1$, then $N' > N$.
\end{corollary}
\begin{proof}
    From \cref{prop:mfmc_var_red}, if $\beta^{\textnormal{MFMC}} < 1$, then $\rho^2 > 4w/(1+w)^2$. 
    Then, from \cref{prop:comp_rho2_thresh}, we deduce that $\rho^2 > 2w / (1+w) > w / (1+w)$. 
    Finally, we conclude from \cref{prop:mfmc_eta} that $\eta > 1$, i.e., $N'>N$.
\end{proof}

\subsection{The exact CV estimator}

We now turn to the exact CV estimator of the expected value of $f(\vX)$, i.e., considering that $\tau := \E[g(\vX)]$ is known. For a given (expected) computational budget $\mathcal{C}$, the (optimal) CV estimator reads
\begin{equation}\label{eq:app_cv_expect}
    \hat{\mu}_{\mathcal{C}}^{\textnormal{CV}}
    =
    \frac{1}{n} \sum_{i=1}^n f(\vX^{(i)}) 
    - 
    \alpha
    \left( \frac{1}{n} \sum_{i=1}^n g(\vX^{(i)}) - \tau \right),
    \quad
    \mbox{ with }
    \alpha = \frac{\C[Y,Z]}{\V[Z]},
    \quad
    n = \frac{\mathcal{C}}{c_f (1+ w)},
\end{equation}
and where $\mathcal{X} := \{\vX^{(1)}, \ldots, \vX^{(n)}\}$ is an $n$-sample of $\vX$.
We have the following properties.

\begin{proposition}\label{prop:cv_beta}
    The variance of the CV estimator \cref{eq:app_cv_expect} is such that
    ${\V[\hat{\mu}_{\mathcal{C}}^{\textnormal{CV}}] = \beta^{\textnormal{CV}} \V[\hat{\mu}_p]}$,
    with
    $\beta^{\textnormal{CV}} = (1+w)(1-\rho^2)$.
\end{proposition}
\begin{proof}
    From \cref{varMCV,app:CVmeanvar-mean}, we know that $\V[\hat{\mu}_{\mathcal{C}}^{\textnormal{CV}}] = (1-\rho^2) \V[\hat{\mu}_n]$.
    The result follows from the fact that $\V[\hat{\mu}_n] = \V[f(\vX)]/n$ and $\V[\hat{\mu}_p] = \V[f(\vX)]/p$, with $n=p/(1+w)$. 
\end{proof}

\begin{proposition}\label{prop:cv_var_red}
    $\beta^{\textnormal{CV}} < 1$ if and only if $\rho^2 > w / (1+w)$.
\end{proposition}
\begin{proof}
    $\beta^{\textnormal{CV}} <1 \iff (1+w)(\rho^2-1) > - 1 \iff \rho^2 > 1 - 1 / (1+w) = w/(1+w)$.
\end{proof}
Combined with \cref{prop:mfmc_var_red,prop:cv_var_red}, \cref{prop:comp_rho2_thresh} indicates that the minimum value of $\rho^2$ that guarantees variance reduction is always (i.e., for any $w \in (0,1)$) at least twice as large for the MFMC estimator than for the CV estimator.
For instance, for $w=1/3$, the MFMC estimator requires $\rho^2 > 3/4$ to have a reduced variance, while the CV estimator only requires $\rho^2>1/4$.

\begin{proposition}\label{prop:comp_beta_cv_mfmc}
    For $w,\rho^2 \in (0,1)$, $\beta^{\textnormal{CV}} < 1$ if and only if $(1+w) \beta^{\textnormal{CV}} < \beta^{\textnormal{MFMC}}$.
\end{proposition}
\begin{proof}
    $\beta^{\textnormal{CV}} < 1 \iff \rho^2 > w / (1+w) \iff 1-\rho^2 < 1/(1+w) \iff 1/(1-\rho^2) > 1 + w$.
    Because $1/(1-\rho^2) -1 = \rho^2 / (1-\rho^2)$,
    $\beta^{\textnormal{CV}} < 1 \iff \rho^2 / (1-\rho^2) > w \iff w\rho^2 / (1-\rho^2) > w^2$.
    We note that $w\rho^2 / (1-\rho^2) = (\eta w)^2$,
    so that
    $\beta^{\textnormal{CV}} < 1 \iff \eta w > w$.
    Therefore, because 
    $\beta^{\textnormal{MFMC}} / \beta^{\textnormal{CV}} = (1+\eta w)^2 / (1+w)$,
    $\beta^{\textnormal{CV}} < 1 \iff \beta^{\textnormal{MFMC}} / \beta^{\textnormal{CV}} > (1+w)^2/(1+w) = 1+w$, which concludes the proof.
\end{proof}

\begin{corollary}\label{coro:comp_beta_cv_mfmc}
    For $w,\rho^2 \in (0,1)$, if $\beta^{\textnormal{CV}} < 1$, then $\beta^{\textnormal{CV}} < \beta^{\textnormal{MFMC}}$.
\end{corollary}
\begin{proof}
    Because $w>0$, 
    $\beta^{\textnormal{MFMC}} / \beta^{\textnormal{CV}} > 1+w \implies \beta^{\textnormal{MFMC}} / \beta^{\textnormal{CV}} > w$.
    Hence the result follows from \cref{prop:comp_beta_cv_mfmc}.
\end{proof}
\Cref{coro:comp_beta_cv_mfmc} implies that, in the $(w, \rho^2)$ variance reduction region of the CV estimator, which is larger than that of the MFMC estimator by \cref{prop:mfmc_var_red,prop:cv_var_red,prop:comp_rho2_thresh}, the variance reduction of the CV estimator is always greater than that of the MFMC estimator (i.e., $\beta^{\textnormal{CV}} < \beta^{\textnormal{MFMC}}$).

\begin{figure}[ht]\centering%
\begin{subfigure}[c]{.48\linewidth}\centering%
  \includegraphics[width=\linewidth]{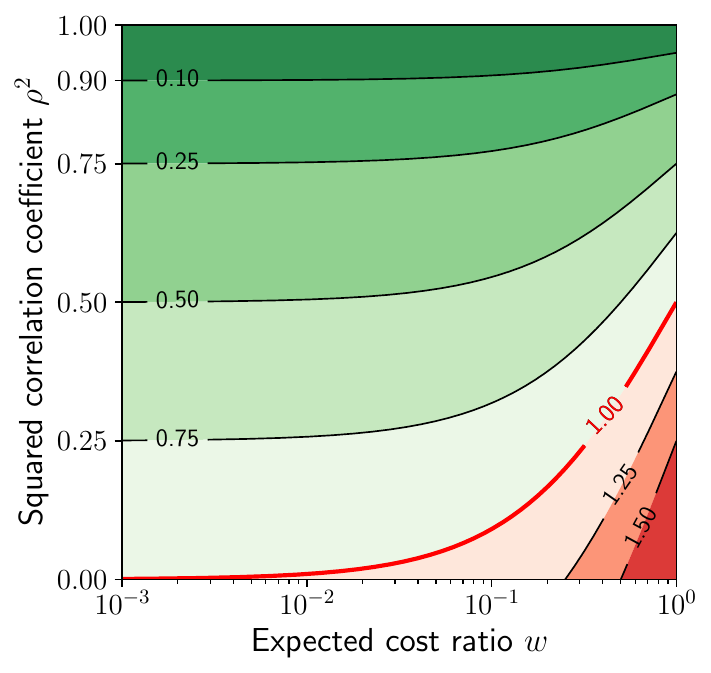}%
  \caption{Standard (exact) CV.}%
  \label{fig:beta_cv}
\end{subfigure}\hfill%
\begin{subfigure}[c]{.48\linewidth}\centering%
  \includegraphics[width=\linewidth]{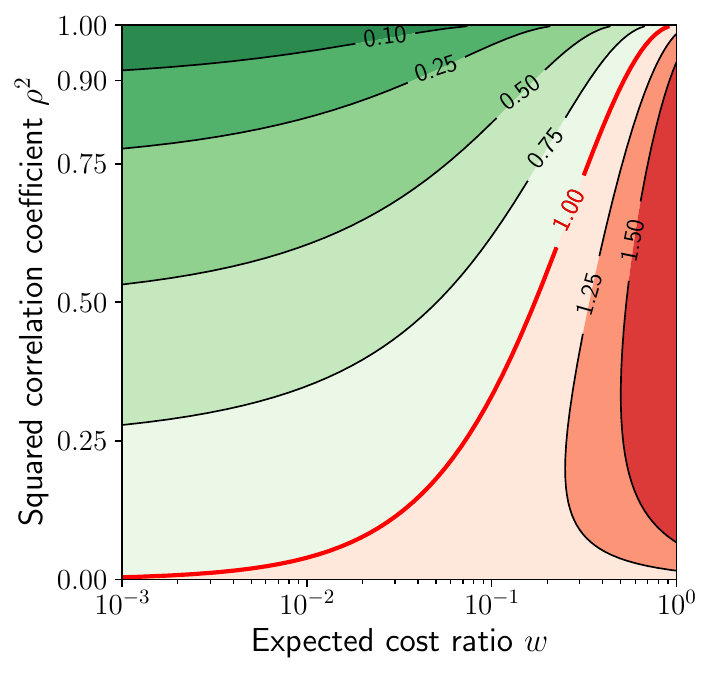}%
  \caption{MFMC.}%
  \label{fig:beta_mfmc}
\end{subfigure}
\caption{Contour lines of $\beta^{\textnormal{CV}}$ (\cref{fig:beta_cv}) and $\beta^{\textnormal{MFMC}}$ (\cref{fig:beta_mfmc}) as functions of $w$ and $\rho^2$.
In each plot, the green region corresponds to pairs $(w, \rho^2)$ for which the variance is reduced, while the red region corresponds to pairs for which the variance is increased.
The boundary between these two regions is represented by a red line.}
\label{fig:beta}
\end{figure}

\Cref{fig:beta} depicts the $(w, \rho^2)$ variance reduction regions for the CV (\cref{fig:beta_cv}) and MFMC (\cref{fig:beta_mfmc}) estimators.
We clearly see that for any $w\in(0,1)$, a higher value of $\rho^2$ is required for the MFMC estimator to achieve variance reduction than for the CV estimator (red solid line).
\Cref{fig:beta} also illustrates an immediate consequence of \cref{prop:cv_beta,prop:mfmc_beta}, namely that $\lim_{w\to 0} \beta^{\textnormal{CV}} = \lim_{w\to 0} \beta^{\textnormal{MFMC}} = 1-\rho^2$.
We also observe that $\beta^{\textnormal{CV}}$ converges to $1-\rho^2$ faster than $\beta^{\textnormal{MFMC}}$, consistent with \cref{coro:comp_beta_cv_mfmc}.

These properties illustrate the advantages of using surrogate-based control variates.
Indeed, when the low-fidelity model $g$ corresponds to a surrogate model, not only $w \ll 1$, but the exact expectation $\E[g(\vX)]$ may be known, so that the exact CV approach can be used, resulting in a variance reduction factor close to $\rho^2$.%

\section{Optimal MLMC-MLCV variance estimator}
\label{app:MLCVmeanvar}

The MLMC-MLCV variance estimator is given by \cref{estimatorMLCV}, with
\begin{align}
    \hat{T}_0^{\textnormal{MLCV}}(\boldsymbol{\alpha}_0)
    & = 
    \hat{V}^{(0)}[Y_0] - \boldsymbol{\alpha}_0^\intercal 
    (\hat{V}^{(0)}[\vZ] - \boldsymbol{\sigma}^2_{\vZ}), \\[.5em]
    \hat{T}_\ell^{\textnormal{MLCV}}(\boldsymbol{\alpha}_\ell)
    & = 
    \hat{V}^{(\ell)}[Y_\ell] - \hat{V}^{(\ell)}[Y_{\ell-1}] 
    - \boldsymbol{\alpha}_\ell^\intercal 
    [
        \hat{V}^{(\ell)}[\vZ_{1:}] - \hat{V}^{(\ell)}[\tilde{\vZ}] 
        - 
        (\boldsymbol{\sigma}^2_{\vZ_{1:}} - \boldsymbol{\sigma}^2_{\tilde{\vZ}})
    ],
    \mbox{ for } \ell>1,
\end{align}
where 
\begin{align}
  \vZ &= (Z_0, \ldots, Z_L) = (g_0(\vX), \ldots, g_L(\vX)), & \boldsymbol{\sigma}^2_\vZ &= \V[\vZ],\\[.5em]
  \vZ_{1:} &= (Z_1, \ldots, Z_L) = (g_1(\vX), \ldots, g_L(\vX)), & \boldsymbol{\sigma}^2_{\vZ_{1:}} &= \V[\vZ_{1:}],\\[.5em]
  \tilde{\vZ} &= (\tilde{Z}_0, \ldots, \tilde{Z}_{L-1}) = (\tilde{g}_0(\vX), \ldots, \tilde{g}_{L-1}(\vX)), & \boldsymbol{\sigma}^2_{\tilde{\vZ}} &= \V[\tilde{\vZ}],\\[.5em]
  \tilde{g}_{\ell-1} &= g_\ell - h_\ell,
  \quad \mbox{for } \ell > 0.
\end{align}
Furthermore, we assume that the expected values of the control variates are known, so that we use the following variance estimators:
\begin{align}
    \hat{V}^{(0)}[\vZ] 
    &= \hat{E}^{(0)}[\bar{\vZ}^2],
    & \hat{V}^{(\ell)}[\vZ_{1:}] - \hat{V}^{(\ell)}[\tilde{\vZ}]
    & = \hat{E}^{(\ell)}[\bar{\vZ}_{1:}^{\odot 2}] - \hat{E}^{(\ell)}[\bar{\tilde{\vZ}}^{\odot 2}]
    =  \hat{E}^{(\ell)}[\bar{\vZ}_{1:}^{\odot 2} - \bar{\tilde{\vZ}}^{\odot 2}].
\end{align}
The optimal values $\boldsymbol{\alpha}_\ell^*$ of the CV parameters are given by
\begin{align}
    \boldsymbol{\alpha}_0^* & = \boldsymbol{\Sigma}_0^{-1} \vec{c}_0,
    & \boldsymbol{\Sigma}_0 & = \C[\bar{\vZ}^{\odot 2}],
    & \vec{c}_0 & = \C[\bar{\vZ}^{\odot 2}, \bar{Y}_0^2], \\
    \boldsymbol{\alpha}_\ell^* & = \boldsymbol{\Sigma}_\ell^{-1} \vec{c}_\ell,
    & \boldsymbol{\Sigma}_\ell & = \C[\bar{\vZ}_{1:}^{\odot 2} - \bar{\tilde{\vZ}}^{\odot 2}],
    & \vec{c}_\ell & = \C[\bar{\vZ}_{1:}^{\odot 2} - \bar{\tilde{\vZ}}^{\odot 2}, \bar{Y}_\ell^2 - \bar{Y}_{\ell-1}^2],
    \quad \mbox{for } \ell >0.
\end{align}
Note that $\boldsymbol{\Sigma}_\ell$ is the same for all $\ell > 0$.
For PC-based control variates, $g_\ell(\vX) = \sum_{k=0}^{P_g^\ell} \mathfrak{g}_{\ell, k} \Psi_k(\vX)$ and $h_\ell(\vX) = \sum_{k=0}^{P_h^\ell} \mathfrak{h}_{\ell, k} \Psi_k(\vX)$, letting $P_{\tilde{g}}^\ell := \max(P_g^\ell, P_h^\ell)$, we have
\begin{equation}
    \tilde{Z}_{\ell-1} = 
    \sum_{k=0}^{P_{\tilde{g}}^\ell} 
    \tilde{\mathfrak{g}}_{\ell-1, k} \Psi_k(\vX),
    \quad \mbox{with }
    \tilde{\mathfrak{g}}_{\ell-1, k} := \mathfrak{g}_{\ell,k} - \mathfrak{h}_{\ell,k},
    \quad \mbox{for } \ell = 1,\ldots,L,
\end{equation}
where 
$\mathfrak{g}_{\ell,k} := 0$ for $k>P_g^\ell$ and $\mathfrak{h}_{\ell,k} := 0$ for $k>P_h^\ell$.
Consequently,
\begin{align}
    \E[Z_\ell] &= \mathfrak{g}_{\ell,0},
    & \E[\tilde{Z}_{\ell-1}] &= \tilde{\mathfrak{g}}_{\ell-1,0}, \\
    \V[Z_\ell] &= \sum_{k=1}^{P_g^\ell} \mathfrak{g}_{\ell,k}^2,
    & \V[\tilde{Z}_{\ell-1}] &= \sum_{k=1}^{P_{\tilde{g}}^\ell} \tilde{\mathfrak{g}}_{\ell-1,k}^2. 
\end{align}
Furthermore,
\begin{equation}
    [\boldsymbol{\Sigma}_0]_{m,m'}
    =
    \sum_{i,j=1}^{P_g^m} \sum_{q,r=1}^{P_g^{m'}}
    \mathfrak{g}_{m,i} \mathfrak{g}_{m,j} \mathfrak{g}_{m',q} \mathfrak{g}_{m',r} 
    (\Phi_{ijqr} - \delta_{ij}\delta_{qr}), 
    \quad \mbox{for } m,m' = 0,\ldots, L,
\end{equation}
where $\Phi_{ijqr} := \E[\Psi_i(\vX)\Psi_j(\vX)\Psi_q(\vX)\Psi_r(\vX)]$ are the entries of the fourth-order Galerkin product tensor, which is a well-known object in stochastic Galerkin methods~\cite{Ghanem2003_StochasticFiniteElements,LeMaitre2010_SpectralMethodsUncertainty}, and $\delta_{ij}$ denotes the Kronecker delta.
Besides, noticing that 
$\bar{\vZ}_{1:}^{\odot 2} - \bar{\tilde{\vZ}}^{\odot 2} 
= (\bar{\vZ}_{1:} - \bar{\tilde{\vZ}}) \odot (\bar{\vZ}_{1:} + \bar{\tilde{\vZ}})
= \bar{\vec{W}} \odot (\bar{\vZ}_{1:} + \bar{\tilde{\vZ}})$, with $\bar{\vec{W}} = \vec{W} - \boldsymbol{\mu}_{\vec{W}}$ and the definitions in \cref{eq:def_W}, we have
\begin{equation}
    [\boldsymbol{\Sigma}_\ell]_{m,m'}
    =
    A_{m,m'} + B_{m,m'} + C_{m,m'} + C_{m',m}, 
    \quad \mbox{for } \ell,m,m' = 1,\ldots, L,
\end{equation}
where
\begin{align}
    A_{m,m'} 
    & =
    \C[\bar{W}_m \bar{Z}_m, \bar{W}_{m'} \bar{Z}_{m'}] \\
    & = 
    \sum_{i=1}^{P_h^m} \sum_{j=1}^{P_g^m} \sum_{q=1}^{P_h^{m'}} \sum_{r=1}^{P_g^{m'}}
    \mathfrak{h}_{m,i} \mathfrak{g}_{m,j} \mathfrak{h}_{m',q} \mathfrak{g}_{m',r} 
    (\Phi_{ijqr} - \delta_{ij}\delta_{qr}), \\[1em]
    B_{m,m'} 
    & =
    \C[\bar{W}_m \bar{\tilde{Z}}_{m-1}, \bar{W}_{m'} \bar{\tilde{Z}}_{m'-1}] \\
    & = 
    \sum_{i=1}^{P_h^m} \sum_{j=1}^{P_{\tilde{g}}^m} \sum_{q=1}^{P_h^{m'}} \sum_{r=1}^{P_{\tilde{g}}^{m'}}
    \mathfrak{h}_{m,i} \tilde{\mathfrak{g}}_{m-1,j} \mathfrak{h}_{m',q} \tilde{\mathfrak{g}}_{m'-1,r} 
    (\Phi_{ijqr} - \delta_{ij}\delta_{qr}), \\[1em]
    C_{m,m'} 
    & =
    \C[\bar{W}_m \bar{\tilde{Z}}_{m-1}, \bar{W}_{m'} \bar{Z}_{m'}] \\
    & = 
    \sum_{i=1}^{P_h^m} \sum_{j=1}^{P_{\tilde{g}}^m} \sum_{q=1}^{P_h^{m'}} \sum_{r=1}^{P_g^{m'}}
    \mathfrak{h}_{m,i} \tilde{\mathfrak{g}}_{m-1,j} \mathfrak{h}_{m',q} \mathfrak{g}_{m',r} 
    (\Phi_{ijqr} - \delta_{ij}\delta_{qr}).
\end{align}


\section{Taylor surrogate for the numerical test case}\label{app:taylor}

In our example,
$f_\ell$ is only differentiable in $[-\pi,\pi]^3 \times [\nu_{\textnormal{min}}, \nu_{\textnormal{max}}] \times ([-1, 0) \cup (0, 1])^3$.
Consequently, the Taylor polynomial surrogate cannot be used directly as defined in~\eqref{eq:Taylor1} and~\eqref{eq:Taylor2}, because the Jacobian and Hessian matrices are not defined at $\boldsymbol{\mu}_\vX$.
Instead, we define the first-order Taylor surrogate as
\begin{equation}\label{eq:Taylor1_num}
    f_\ell(\vX) \simeq g_\ell^{\textnormal{T}_1}(\vX) 
    = 
    f_\ell(\boldsymbol{\mu}_\vX) 
    + 
    \sum_{i=1}^7 g_{\ell, i}^{\textnormal{T}_1}(\boldsymbol{\mu}_\vX; \vX),
\end{equation}
where, for $i=1,\ldots,4$, 
$g_{\ell, i}^{\textnormal{T}_1}(\boldsymbol{\mu}_\vX; \vX)
=
(X_i - \mu_{X_i})
\dfrac{\partial f_\ell}{\partial X_i}(\boldsymbol{\mu}_\vX)
$, and, for $i=5,\ldots,7$,
\begin{equation}
    g_{\ell, i}^{\textnormal{T}_1}(\boldsymbol{\mu}_\vX; \vX)
    =  (X_i - \mu_{X_i})
    \times 
    \begin{cases}
        \dfrac{\partial f_\ell}{\partial X_i}(\boldsymbol{\mu}_\vX),
        & \mu_{X_i} \ne 0, \\[1em]
        \lim_{\boldsymbol{\mu}_\vX' \to \boldsymbol{\mu}_\vX^{i, 0^-}}\dfrac{\partial f_\ell}{\partial X_i}(\boldsymbol{\mu}_\vX'),
        & \mu_{X_i} = 0, X_i <0,\\[1em]
        \lim_{\boldsymbol{\mu}_\vX' \to \boldsymbol{\mu}_\vX^{i, 0^+}}\dfrac{\partial f_\ell}{\partial X_i}(\boldsymbol{\mu}_\vX'),
        & \mu_{X_i} = 0, X_i >0,\\[1em]
        0 &  \mu_{X_i} = X_i = 0,
    \end{cases}
\end{equation}
where $\boldsymbol{\mu}_\vX^{i, 0^\pm} = (\mu_{X_1}, \ldots, \mu_{X_{i-1}}, 0^\pm, \mu_{X_{i+1}}, \ldots, \mu_{X_7})$, which is now well-defined.

With our choice of distributions for $\vX$ given in~\eqref{eq:distrib}, we have $\boldsymbol{\mu}_\vX = (0,0,0,0.005,0,0,0)$, and the first-order Taylor surrogate is defined by~\eqref{eq:Taylor1_num}, with
\begin{align}
    g_{\ell, 1}^{\textnormal{T}_1}(\boldsymbol{\mu}_\vX; \vX) &
    = 7 X_1 
    \sum_{k=1}^K B_k^\ell(T; \boldsymbol{\mu}_\vX)
    \sum_{j=1}^{N_\ell} w_j \sin(k\pi x_j) \mathcal{F}_2(x_j), \\
    g_{\ell, 2}^{\textnormal{T}_1}(\boldsymbol{\mu}_\vX; \vX) & = 0,\\
    g_{\ell, 3}^{\textnormal{T}_1}(\boldsymbol{\mu}_\vX; \vX) & = 0,\\
    g_{\ell, 4}^{\textnormal{T}_1}(\boldsymbol{\mu}_\vX; \vX) &
    = - (X_4 - 0.005) \pi^2 T 
    \sum_{k=1}^K k^2 \exp(-0.005 k^2\pi^2 T) A_k^\ell(\boldsymbol{\mu}_\vX)
    \sum_{j=1}^{N_\ell} w_j \sin(k\pi x_j),\\
    g_{\ell, 5}^{\textnormal{T}_1}(\boldsymbol{\mu}_\vX; \vX) &
    = 400 |X_5| 
    \sum_{k=1}^K B_k^\ell(T; \boldsymbol{\mu}_\vX)
    \sum_{j=1}^{N_\ell} w_j \sin(k\pi x_j) \mathcal{F}_1(x_j),\\
    g_{\ell, 6}^{\textnormal{T}_1}(\boldsymbol{\mu}_\vX; \vX) &
    = 400 |X_6| 
    \sum_{k=1}^K B_k^\ell(T; \boldsymbol{\mu}_\vX)
    \sum_{j=1}^{N_\ell} w_j \sin(k\pi x_j) \mathcal{F}_1(x_j),\\
    g_{\ell, 7}^{\textnormal{T}_1}(\boldsymbol{\mu}_\vX; \vX) &
    = 400 |X_7| 
    \sum_{k=1}^K B_k^\ell(T; \boldsymbol{\mu}_\vX)
    \sum_{j=1}^{N_\ell} w_j \sin(k\pi x_j) \mathcal{F}_1(x_j).
\end{align}
Because of the piecewise definition of $g_\ell^{\textnormal{T}_1}$, 
the identity $\E[g_\ell^{\textnormal{T}_1}(\vX)] = f_\ell(\boldsymbol{\mu}_{\vX})$ for a regular first-order Taylor surrogate no longer holds. Instead, we have
\begin{equation}
    \E[g_\ell^{\textnormal{T}_1}(\vX)] 
    =
    f_\ell(\boldsymbol{\mu}_{\vX})
    +
    600 \sum_{k=1}^K B_k^\ell(T; \boldsymbol{\mu}_\vX)
    \sum_{j=1}^{N_\ell} w_j \sin(k\pi x_j) \mathcal{F}_1(x_j).
\end{equation}

\bibliographystyle{plain}
\bibliography{references}

\end{document}